\documentclass[11pt]{amsart}
\usepackage{palatino}
\usepackage{amsfonts}
\usepackage{amssymb}
\usepackage{amscd}
\input xy
\xyoption{all}
\usepackage[english]{babel}
\usepackage[all]{xy}
\usepackage[breaklinks,bookmarksopen,bookmarksnumbered]{hyperref}

\newtheorem{theorem}{Theorem}[section]
\newtheorem{proposition}[theorem]{Proposition}

\newtheorem{lemma}[theorem]{Lemma}
\theoremstyle{definition}
\newtheorem{definition}[theorem]{Definition}
\newtheorem{remark}[theorem]{Remark}
\newtheorem{example}[theorem]{Example}

\newtheorem{conjecture}[theorem]{Conjecture}
\newtheorem{conjecture/question}[theorem]{Conjecture/Question}

\newtheorem{remark/definition}[theorem]{Remark/Definition}
\newtheorem{terminology/notation}[theorem]{Terminology/Notation}

\newcommand{\marginlabel}[1]%
  {\mbox{}\marginpar{\raggedleft\hspace{0pt}\bfseries\sf#1}}

\def\PP{{\textbf P}}
\def\OO{\mathcal{O}}
\def\cN{\mathcal{N}}
\def\cD{\mathcal{D}}
\def\cB{\mathcal{B}}
\def\cA{\mathcal{A}}
\def\F{\mathcal{F}}
\def\P{\mathcal{P}}

\def\E{\mathcal{E}}
\def\G{\mathcal{G}}

\def\L{\mathcal{L}}

\def\cM{\mathcal{M}}
\def\cR{\mathcal{R}}
\def\rr{\overline{\mathcal{R}}}
\def\cZ{\mathcal{Z}}
\def\cU{\mathcal{U}}

\def\H{\mathcal{H}}
\def\Pic0{{\rm Pic}^0(X)}

\def\mm{\overline{\mathcal{M}}}

\def\pr{\widetilde{\mathcal{R}}}
\def\AUX{\mathfrak G^r_d(\rer_g^0/\rem_g^0)}
\def\rem{\overline{\textbf{M}}}
\def\rer{\overline{\textbf{R}}}
\def\per{\widetilde{\textbf{R}}}
\def\pem{\widetilde{\textbf{M}}}

\newcommand{\CC}[1][]{\mathbb{C}^{#1}}

\DeclareMathOperator{\aut}{Aut} \DeclareMathOperator{\auto}{Aut_0}

\DeclareMathOperator{\ord}{ord} 

\newcommand{\comp}{C_j}
\newcommand{\compi}{C_{j'}}
\newcommand{\compii}{C_{j''}}
\newcommand{\compnu}{C_j^\nu}
\newcommand{\compinu}{C_{j'}^\nu}
\newcommand{\iso}{\sigma}
\newcommand{\lud}[1]{\CC[3g-3]_{#1}}
\newcommand{\nonex}{\widetilde{\prymi}}
\newcommand{\nonexcomp}{\widetilde{\prymi}_j}

\newcommand{\prym}{(\prymi,\prymii,\prymiii)}
\newcommand{\prymi}{X}
\newcommand{\prymii}{\eta}
\newcommand{\prymiii}{\beta}
\newcommand{\rsbt}{Reid--Shepherd-Barron--Tai}
\newcommand{\stab}{C}

\newcounter{mylist}
{\begin{list}%
    {(\roman{mylist})}%
    {\usecounter{mylist}%
     \setlength{\rightmargin}{0pt}%
     \setlength{\leftmargin}{0pt}%
     \setlength{\itemindent}{0.5em}%
     \setlength{\itemsep}{0pt}%
     \setlength{\parsep}{0ex plus0.1ex}
     \setlength{\topsep}{0ex}}}%
{\end{list}}

\begin{document}
\title{The Kodaira dimension of the moduli space of Prym varieties}

\author[G. Farkas]{Gavril Farkas}

\address{Humboldt-Universit\"at zu Berlin, Institut F\"ur Mathematik,
10099 Berlin}
\email{{\tt farkas@math.hu-berlin.de}}
\thanks{Research  of the first author partially supported by an Alfred P. Sloan Fellowship, the NSF Grant DMS-0500747
and a Texas Research Assignment}

\author[K. Ludwig]{Katharina Ludwig}
\address{Leibniz Universit\"at Hannover, Institut f\"ur Algebraische
Geometrie
\newline \indent D-30167 Hannover, Germany}
 \email{{\tt
ludwig@math.uni-hannover.de}}

\markboth{G. FARKAS and K. LUDWIG}
{The Kodaira dimension of $\rr_g$}
\maketitle

Prym varieties provide a correspondence between the moduli spaces of
curves and abelian varieties $\cM_g$ and $\cA_{g-1}$, via the
\emph{Prym map} $\mathcal{P}_g: \cR_g\rightarrow \cA_{g-1}$ from the
moduli space $\cR_g$ parameterizing pairs $[C, \eta]$, where $[C]\in
\cM_g$ is a smooth curve and $\eta \in \mbox{Pic}^0(C)[2]$ is a
torsion point of order $2$. When $g\leq 6$ the Prym map is dominant
and $\cR_g$ can be used directly to determine the birational type of
$\cA_{g-1}$. It is known that $\cR_g$ is rational for $g=2, 3, 4$
(see \cite{Dol} and references therein and \cite{Ca} for the case of genus $4$) and unirational for $g=5$
(cf. \cite{IGS} and \cite{V2}). The situation in genus $6$ is strikingly beautiful
because $\mathcal{P}_6:\cR_6\rightarrow \cA_5$ is equidimensional
(precisely $\mbox{dim}(\cR_6)=\mbox{dim}(\cA_5)=15$). Donagi and
Smith showed that $\P_6$ is generically finite of degree $27$ (cf.
\cite{DS}) and the monodromy group equals the Weyl group $WE_6$
describing the incidence correspondence of the $27$ lines on a cubic
surface (cf. \cite{D1}). There are three different proofs that
$\cR_6$ is unirational (cf. \cite{D1}, \cite{MM}, \cite{V}). Verra
has very recently announced a proof of the unirationality of
$\cR_7$ (see also Theorem \ref{r7} for a weaker result).
The Prym map $\mathcal{P}_g$ is generically injective for
$g\geq 7$ (cf. \cite{FS}), although never injective. In this range,
we may regard $\cR_g$ as a partial desingularization of the moduli
space $\mathcal{P}_g(\cR_g)\subset \cA_{g-1}$ of Prym varieties of
dimension $g-1$.

The scheme $\cR_g$ admits a suitable modular compactification
$\rr_g$, which is isomorphic to (1) the coarse moduli space of the
stack $\rer_g=\rem_g(\mathcal{B}\mathbb Z_2)$ of \emph{Beauville admissible} double covers (cf. \cite{B}, \cite{ACV})
and (2) the coarse moduli space of the stack of \emph{Prym curves}
(cf. \cite{BCF}). The forgetful map $\pi:\cR_g \rightarrow \cM_g$
extends to a finite map $\pi: \rr_g \rightarrow \mm_g$.  The aim of
this paper is to initiate a study of the enumerative and global
geometry of $\rr_g$, in particular to determine its Kodaira
dimension. The main result of the paper is the following:
\begin{theorem}\label{gentype}
The moduli space of Prym varieties $\rr_g$ is of general type for
$g>13$ and $g\neq 15$. The Kodaira dimension of $\rr_{15}$ is at
least $1$.
\end{theorem}
We point out in Remark \ref{sl15} that the existence of an effective
divisor $D\in \mbox{Eff}(\mm_{15})$ of slope $s(D)<6+12/(g+1)=27/4$
(that is, violating the Harris-Morrison Slope Conjecture on
$\mm_{15}$), would imply that $\rr_{15}$ is of general type. There
are known examples of divisors $D\in \mathrm{Eff}(\mm_g)$ satisfying
$s(D)<6+\frac{12}{g+1}$ for every genus of the form $g=s(2s+si+i+1)$ with $s\geq 2$ and
$i\geq 0$ (cf. \cite{F1}, \cite{F2}). No such examples have been
found yet on $\mm_{15}$, though they are certainly expected to exist.

The normal variety $\rr_g$ has finite quotient singularities and an
important part of the proof is concerned with showing that  pluricanonical forms defined on the smooth
part $\rr_g^{\mathrm{reg}}\subset \rr_g$ can be lifted to any
resolution of singularities $\widehat{\mathcal{R}}_g\rightarrow
\rr_g$, that is, we have isomorphisms $$H^0\bigl(
\rr_g^{\mathrm{reg}}, K_{\rr_g}^{\otimes l}\bigr)\cong
H^0\bigl(\widehat{\mathcal{R}}_g,
K_{\widehat{\mathcal{R}}_g}^{\otimes l}\bigr)$$ for $l\geq 0$. This is achieved in the last section of the paper.  The locus of non-canonical singularities in $\rr_g$ is also
explicitly described: A Prym curve $[X, \eta, \beta]\in \rr_g$ is a
non-canonical singularity if and only if $X$ has an elliptic tail
$C$ with $\mbox{Aut}(C)=\mathbb Z_6$, such that the line bundle
$\eta_C\in \mbox{Pic}^0(C)[2]$ is trivial (cf. Theorem
\ref{thm:cansing}).

We outline the strategy to prove that $\rr_g$ is of general type for
given $g$. If $\lambda=\pi^*(\lambda)\in \mbox{Pic}(\rr_g)$ is the
pull-back of the Hodge class and $\delta_0', \delta_0^{''},
\delta_0^{\mathrm{ram}}\in \mbox{Pic}(\rr_g)$ and  $\delta_{i},
\delta_{g-i}, \delta_{i: g-i}\in \mbox{Pic}(\rr_g)$ for $1\leq i\leq
[g/2]$ are boundary divisor classes such that
$$\pi^*(\delta_0)=\delta_0'+\delta_0^{''}+2\delta_{0}^{\mathrm{ram}}
\\ \mbox{ and } \
\pi^*(\delta_i)=\delta_i+\delta_{g-i}+\delta_{i: g-i}\ \mbox{ for }
1\leq i\leq [g/2]$$ (see Section 2 for a precise definition of these
classes), then one has the formula
$$K_{\rr_g}\equiv 13\lambda-2(\delta_0^{'}+\delta_0^{''})-3\delta_0^{\mathrm{ram}}-2\sum_{i=1}^{[g/2]}
(\delta_i+\delta_{g-i}+\delta_{i:
g-i})-(\delta_1+\delta_{g-1}+\delta_{1: g-1}).$$  We show that this
class is big by explicitly constructing effective divisors $D$ on
$\rr_g$ such that one can write $K_{\rr_g}\equiv \alpha\cdot
\lambda+\beta \cdot D+\{\mbox{effective combination of boundary
classes}\}$, for certain $\alpha, \beta\in \mathbb Q_{>0}$ (see
(\ref{inequ}) for the  inequalities the coefficients of such $D$
must satisfy).

We carry out an enumerative study of divisors on $\rr_g$ defined in
terms of pairs $[C, \eta]$ such that the $2$-torsion point $\eta\in
\mbox{Pic}^0(C)$ is transversal with respect to the theta divisors
associated to certain stable vector bundles on $C$. We fix integers
$k\geq 2$ and $b\geq 0$ and then define the integers $$i:=kb+k-b-2,\
r:=kb+k-2,\ g:=ik+1\ \mbox{ and } d:=rk.$$ The Brill-Noether number
$\rho(g, r, d):=g-(r+1)(g-d+r)=0$ and a general $[C]\in \cM_g$
carries a finite number of line bundles $L\in W^r_d(C)$. For each
such line bundle $L$, if $Q_L$ denotes the dual of the
\emph{Lazarsfeld bundle} defined by the exact sequence (see
\cite{L})
$$0\longrightarrow Q_L^{\vee}\longrightarrow H^0(C, L)\otimes \OO_C
\longrightarrow L\longrightarrow 0,$$ we compute that
$\mu(Q_L)=d/r=k$ and then $\mu(\wedge^i Q_L)=ik=g-1$. In these
circumstances we define the \emph{Raynaud divisor} (degeneration
locus of virtual codimension $1$)
$$\Theta_{\wedge^i Q_L}:=\{\eta \in \mathrm{Pic}^0(C): H^0(C,
\wedge^i Q_L\otimes \eta)\neq 0\}.$$ This is a virtual divisor
inside $\mbox{Pic}^0(C)$, in the sense that  either
$\Theta_{\wedge^i Q_L}=\mbox{Pic}^0(C)$ or else $\Theta_{\wedge^i
Q_L}$ is a divisor on $\mbox{Pic}^0(C)$ belonging to the linear
system $|{r\choose i}\theta|$, cf. \cite{R}. We study the relative
position of $\eta$ with respect to $\Theta_{\wedge^i Q_L}$ and
introduce the following locus on $\rr_g$:
$$\mathcal{D}_{g: k}:=\{[C, \eta]\in \cR_g: \exists L\in W^r_{d}(C)\
\mbox{ such that } \ \eta \in \Theta_{\wedge^i Q_L}\}.$$

When $k=2, i=b$, then $g=2i+1, d=2g-2$ and  $\mathcal{D}_{2i+1: 2}$
has a new incarnation using the proof of the \emph{Minimal
Resolution Conjecture} \cite{FMP}. In this case, $L=K_C$ (a genus
$g$ curve has only one $\mathfrak g^{g-1}_{2g-2}$!) and \cite{FMP}
gives an identification of cycles
$$\Theta_{\wedge^i Q_{K_C}}=C_i-C_i\subset \mbox{Pic}^0(C),$$ where the right-hand-side
stands for the $i$-th difference variety of $C$.

We prove in Section 2 that $\mathcal{D}_{g: k}$ is an effective
divisor on $\mathcal{R}_g$. By specialization to the $k$-gonal locus
$\cM_{g, k}^1\subset \cM_g$, we show that for a generic $[C,
\eta]\in \mathcal{R}_g$  the vanishing $H^0(C, \wedge^i Q_L\otimes
\eta)=0$ holds for every $L\in W^r_d(C)$ (Theorem
\ref{transversality}). Then we extend the determinantal structure of
$\mathcal{D}_{g: k}$ to a partial compactification of
$\mathcal{R}_g$ which enables us to compute the class of the
compactification $\overline{\mathcal{D}}_{g: k}$. Precisely we
construct two vector bundles $\E$ and $\F$ over a stack $\rer_g^0$ which is a partial compactification
of $\textbf{R}_g$, such that
$\mbox{rank}(\E)=\mbox{rank}(\F)$, together with a vector bundle
homomorphism $\phi:\E\rightarrow \F$ such that $Z_1(\phi)\cap
\cR_g=\mathcal{D}_{g: k}$. Then we explicitly determine the class
$c_1(\F-\E)\in A^1(\rer_g^0)$ (Theorem \ref{det}). The cases of
interest for determining the Kodaira dimension of $\rr_g$ are when
$k=2, 3$ when we obtain the following results:

\begin{theorem}\label{hyper}
The closure of the divisor $\mathcal{D}_{2i+1: 2}=\{[C, \eta]\in
\cR_{2i+1}:h^0(C, \wedge^i Q_{K_C}\otimes \eta)\geq 1\}$ inside
$\rr_{2i+1}$ has class given by the following formula in
$\mathrm{Pic}(\rr_{2i+1})$:
$$\overline{\cD}_{2i+1: 2}\equiv \frac{1}{2i-1}{2i\choose
i}\Bigl((3i+1)\lambda-\frac{i}{2}(\delta_0'+\delta_0^{''})-\frac{2i+1}{4}\delta_0^{\mathrm{ram}}
-(3i-1)\delta_{g-1}-i(\delta_{1: g-1}+\delta_1)-\cdots\Bigr).$$
\end{theorem}
To illustrate Theorem \ref{hyper} in the simplest case, $i=1$ hence
$g=3$, we write  $\mathcal{D}_{3: 2}=\{[C, \eta]\in
\cR_3:\eta=\OO_C(x-y), \mbox{ } x, y\in C\}$. The analysis carried
out in Section 5 shows that the vector bundle morphism
$\phi:\E\rightarrow \F$ is generically non-degenerate along the
boundary divisors $\Delta_0', \ \Delta_0^{\mathrm{ram}}\subset
\rr_3$ and degenerate (with multiplicity $1$) along the divisor
$\Delta_0^{''}\subset \rr_3$ of Wirtinger covers. Theorem
\ref{hyper} reads like
$$\overline{\mathcal{D}}_{3: 2}\equiv c_1(\F-\E)-\delta_0^{''}\equiv
8\lambda-\delta_0'-2\delta_0^{''}-\frac{3}{2}\delta_0^{\mathrm{ram}}-6\delta_1-4\delta_2-2\delta_{1:
2}
 \in \mathrm{Pic}(\rr_3),$$
and then $\pi_*(\overline{\mathcal{D}}_{3:
2})\equiv 56(9\lambda-\delta_0-3\delta_1)\equiv 56\cdot \mm_{3, 2}^1\in
\mathrm{Pic}(\mm_3)$ (see Theorem \ref{genu3}). Theorem \ref{hyper}
is consistent with the following elementary fact, see e.g.
\cite{HF}: If $[\tilde{C}\rightarrow C]\in \cR_3$ is an \'etale
double cover, then $[\widetilde{C}]\in \cM_5$ is hyperelliptic if
and only if $[C]\in \cM_3$ is hyperelliptic and $\eta=\OO_C(x-y)$,
with $x, y\in C$ being Weierstrass points.

\begin{theorem}\label{trig}
For $b\geq 1$ and $r=3b+1$ the class of the divisor
$\overline{\cD}_{6b+4: 3}$ on $\rr_{6b+4}$ is given by:
$$\overline{\cD}_{g: 3}\equiv \frac{4}{r}{6b+3\choose
b, 2b,
3b+3}\Bigl((3b+2)(b+2)\lambda-\frac{3b^2+7b+3}{6}(\delta_0'+\delta_0^{''})-\frac{24b^2+47b+21}{24}\delta_0^{\mathrm{ram}}
-\cdots \Bigr).$$
\end{theorem}
Theorems \ref{det}, \ref{hyper} and \ref{trig} specify precisely the
$\lambda, \delta_0', \delta_0^{'}$ and $\delta_0^{\mathrm{ram}}$
coefficients in the expansion of $[\overline{\mathcal{D}}_{g: k}]$.
Good lower bounds for the remaining boundary coefficients of
$[\overline{\mathcal{D}}_{g: k}]$ can be obtained using Proposition
\ref{ineq}. The information contained in Theorems \ref{hyper} and
\ref{trig} is sufficient to finish the proof of Theorem
\ref{gentype} for odd genus $g=2i+1\geq 15$.

When $b=0$, hence $i=r=k-2$, Theorem \ref{det} has the following interpretation:
\begin{theorem}\label{be0}
We fix integers $k\geq 3, r=k-2$ and $g=(k-1)^2$. The following
locus $$\mathcal{D}_{g:k}:=\{[C, \eta]\in \cR_{g}:\exists L\in
W^{k-2}_{k(k-2)}(C)\  \mbox{ such that } H^0(C, L\otimes \eta)\neq
0\}$$ is a divisor on $\cR_g$. The class of its compactification
inside $\rr_g$ is given by the formula
$$\overline{\mathcal{D}}_{g:k}\equiv g!\frac{1!\ 2!\ \cdots (k-2)!}{(k-1)!\
\cdots (2k-3)!\ (k^2-2k-1)} \Bigl(\frac{1}{2}(k^4-4k^3+11k^2-14k+2)\lambda-$$
$$-\frac{k(k-2)(k^2-2k+5)}{12}(\delta_0'+\delta_0^{''})-
\frac{(k^2-2k+3)(2k^2-4k+1)}{12}\delta_0^{\mathrm{ram}}-\cdots\Bigr)\in
\mathrm{Pic}(\rr_g).
$$
\end{theorem}
When $k=3$ and $g=4$, the divisor $\mathcal{D}_{4: 3}$ consists of
Prym curves $[C, \eta]\in \cR_4$ for which the plane
Prym-canonical model $\iota:C\stackrel{|K_C\otimes \eta|}\longrightarrow
\PP^2$ has a triple point. Note that for a general $[C, \eta]\in \cR_4$, $\iota(C)$ is a $6$-nodal sextic. We can then verify the formula
$$\pi_*(\overline{\mathcal{D}}_{4: 3})=60(34\lambda-4\delta_0-14\delta_1-18\delta_2)=60\cdot
\overline{\mathcal{GP}}_{4, 3}^1\in \mathrm{Pic}(\mm_4),$$ where
$\overline{\mathcal{GP}}_{4, 3}^1\subset \mm_4$ is the divisor of
curves with a vanishing theta-null. This is consistent with the
set-theoretic equality $\pi(\mathcal{D}_{4: 3})=\mathcal{GP}_{4,
3}^1$ which can be easily established (see Theorem \ref{genu4}).

Another case which has a simple interpretation is when $b=1$,
$i=r-1$, and then $g=(2k-1)(k-1), d=2k(k-1)$. Since
$\mbox{rank}(Q_L)=r$ and $\mbox{det}(Q_L)=L$, by duality we have that
$\wedge^i Q_L=M_L\otimes L$, hence points $[C, \eta]\in
\mathcal{D}_{(2k-1)(k-1): k}$ can be described purely in terms of
multiplication maps of sections of line bundles on $C$:
\begin{theorem}
We fix integers $k\geq 2$ and $g=(2k-1)(k-1)$. The following locus
$$\mathcal{D}_{g: k}=\{[C, \eta]\in \cR_g: \exists L\in
W_{2k(k-1)}^{2k-2}(C) \mbox{ with } H^0(L)\otimes H^0(L\otimes
\eta)\rightarrow H^0(L^{\otimes 2}\otimes \eta) \ \mbox{ not
bijective}\}$$ is a divisor on $\cR_{g}$. The class of its
compactification inside $\rr_{g}$ equals
$$\overline{\mathcal{D}}_{g: k}\equiv
g!\frac{1!\ 2! \cdots (k-2)!\ (k-1)}{3(2k^2-3k-1)(2k-1)!\ (2k)!\cdots
(3k-3)!\ (3k-2)}\cdot $$
$$\Bigl(6\bigl(8k^5-36k^4+78k^3-95k^2+49k-6\bigr)\lambda-\bigl(8k^5-36k^4+70k^3-71k^2+29k-2\bigr)\bigl(\delta_0'+\delta_0^{''}\bigr)-$$
$$-\frac{1}{2}\bigl(32k^5-144k^4
+262k^3-245k^2+107k-14\bigr)\delta_0^{\mathrm{ram}}-\cdots\Bigr).$$
\end{theorem}

The second class of (virtual) divisors is provided by Koszul
divisors on $\rr_g$. For a pair $(C, L)$ consisting of a curve
$[C]\in \cM_g$ and a line bundle $L\in \mbox{Pic}(C)$, we denote by
$K_{i, j}(C, L)$ its $(i, j)$-th Koszul cohomology group, cf. \cite{L}.  For a
point $[C, \eta]\in \cR_g$ we set $L:=K_C\otimes \eta$ and we
stratify $\cR_g$ using the syzygies of the Prym-canonical curve
$C\stackrel{|L|}\rightarrow \PP^{g-2}$. We define the stratum
$$\mathcal{U}_{g,
i}:=\{[C, \eta]\in \cR_g: K_{i, 2}(C, K_C\otimes \eta)\neq
\emptyset\},$$ that is, $\mathcal{U}_{g, i}$ consists of those Prym
curves $[C, \eta]\in \cR_g$ for which the Prym-canonical model
$C\stackrel{|L|}\longrightarrow \PP^{g-2}$ fails to satisfy the
Green-Lazarsfeld property $(N_i)$ in the sense of \cite{GL},
\cite{L}.

\begin{theorem}\label{prymkoszul}
There exist two vector bundles $\G_{i, 2}$ and $\H_{i, 2}$ of the
same rank defined over a partial compactification $\per_{2i+6}$ of
the stack
$\textbf{R}_{2i+6}$, together with a morphism $\phi:\H_{i, 2}\rightarrow
\G_{i, 2}$ such that
$$\cU_{2i+6, i}:=\{[C, \eta]\in \widetilde{\cR}_{2i+6}: K_{i, 2}(C, K_C\otimes \eta)\neq 0\}$$
is the degeneracy locus of the map $\phi$. The virtual class of
$[\overline{\cU}_{2i+6, i}]$ is given by the formula:
$$[\overline{\cU}_{2i+6, i}]^{virt}=c_1(\G_{i, 2}-\H_{i, 2})={2i+2\choose
i}\Bigl(\frac{3(2i+7)}{i+3}\lambda-\frac{3}{2}\delta_0^{\mathrm{ram}}-(\delta_0'+\alpha\
\delta_0^{''})-\cdots \Bigr),$$ where the constant $\alpha$
satisfies $\alpha\geq 1$.
\end{theorem}

The compactification $\widetilde{\textbf{R}}_{g}$ has the property that if $\pr_g\subset \rr_g$ denotes its coarse moduli space, then $\mbox{codim}\bigl(\pi^{-1}(\cM_g\cup \Delta_0)-\pr_g\bigr)\geq 2$. In particular Theorem \ref{prymkoszul} precisely determines the coefficient of $\lambda, \delta_0', \delta_0^{''}$ and $\delta_0^{\mathrm{ram}}$ in the expansion of $[\overline{\cU}_{2i+6, i}]^{virt}$. We also show that when $g<2i+6$ then $K_{i, 2}(C, K_C\otimes
\eta)\neq \emptyset$ for any $[C, \eta]\in \cR_{g}$. By analogy with
the case of canonical curves and the classical  M. Green Conjecture
on syzygies of canonical curves (see \cite{Vo}), we conjecture that
the morphism of vector bundles  $\phi: \G_{i, 2}\rightarrow \H_{i,
2}$ over $\widetilde{\textbf{R}}_{2i+6}$ is generically non-degenerate:
\begin{conjecture}\label{prymgreen}
\emph{(Prym-Green Conjecture)} For a generic point $[C, \eta]\in
\cR_g$ and $g\geq 2i+6$, we have that $K_{i, 2}(C, K_C\otimes
\eta)=0$. Equivalently, the Prym-canonical curve
$C\stackrel{|K_C\otimes \eta|}\hookrightarrow \PP^{g-2}$ satisfies
the Green-Lazarsfeld property $(N_i)$ whenever $g\geq 2i+6$. For
$g=2i+6$, the locus $\cU_{2i+6, i}$ is an effective divisor on
$\cR_{2i+6}$.
\end{conjecture}

Proposition \ref{ni} shows that, if true, Conjecture \ref{prymgreen}
is sharp. In \cite{F4} we verify the Prym-Green Conjecture for
$g=2i+6$ with $0\leq i\leq 4, i\neq 1$. In particular, this together with
Theorem \ref{prymkoszul} proves that $\rr_g$ is of general type for
$g=14$.

The strata $\cU_{g, i}$ have been considered before for $i=0, 1$, in
connection with the Prym-Torelli problem. Unlike the classical
Torelli problem, the Prym-Torelli problem is a subtle question:
Donagi's tetragonal construction shows that $\mathcal{P}_g$ fails to
be injective over points $[C, \eta]\in \pi^{-1}(\cM_{g, 4}^1)$ where
the curve $C$ is tetragonal (cf. \cite{D2}). The locus $\cU_{g, 0}$
consists of those points $[C, \eta]\in \cR_g$ where the differential
$$(d\mathcal{P}_g)_{[C, \eta]}: H^0(C, K_C^{\otimes
2})^{\vee}\rightarrow (\mbox{Sym}^2 H^0(C, K_C\otimes
\eta))^{\vee}$$ is not injective and thus the \emph{infinitesimal
Prym-Torelli theorem} fails. It is known that $(d\mathcal{P}_g)_{[C,
\eta]}$ is generically injective for $g\geq 6$ (cf. \cite{B}, or
\cite{De} Corollaire 2.3), that is, $\cU_{g, 0}$ is a proper
subvariety of $\cR_g$. In particular, for $g=6$ the locus $\cU_{6,
0}$ is a divisor of $\cR_6$, which gives another proof of Conjecture
\ref{prymgreen} for $i=0$.

Debarre proved that $\cU_{g, 1}$ is a proper subvariety of $\cR_g$
for $g\geq 9$ (cf. \cite{De} Th\'eoreme 2.2). This immediately
implies that for $g\geq 9$ the Prym map $\mathcal{P}_g$ is
generically injective, hence the \emph{Prym-Torelli theorem} holds
generically. Debarre's proof unfortunately does not cover the
interesting case $g=8$.

The proof of Theorem \ref{gentype} is finished in Section 4, using
in an essential way results from \cite{F3}: We set
$g':=1+\frac{g-1}{g}{2g\choose g-1}$ and then we consider the
rational map which associates to a curve one of its Brill-Noether loci
$$\phi: \mm_{2g-1}\dashrightarrow 
\mm_{1+\frac{g-1}{g}{2g\choose g-1}}, \ \mbox{
 }\  \phi[Y]:=W^1_{g+1}(Y),$$
where $W^1_{g+1}(Y):=\{L\in \mathrm{Pic}^{g+1}(Y): h^0(Y, L)\geq 2\}$.
If $\chi:\rr_g \rightarrow \mm_{2g-1}$ is the map given
by $\chi([C, \eta]):=[\tilde{C}]$, where $f:\tilde{C}\rightarrow C$
is the \'etale double cover with the property that
$f_*\OO_{\tilde{C}}=\OO_C\oplus \eta$, then using \cite{F3} we
compute the slope of myriads of effective divisors of type
$\chi^*\phi^*(A)$, where $A\in \mbox{Ample}(\mm_{g'})$. This proves
Theorem \ref{gentype} for even genus $g=2i+6\geq 18$.

We mention in passing as an  immediate application of Proposition
\ref{ineq}, a different proof of the statement that $\rr_g$ has good rationality  properties for low $g$
(see again the Introduction for the history of this problem). Our proof is quite simple and uses only  numerical properties of Lefschetz pencils of curves on $K3$ surfaces:
\begin{theorem}\label{r7}
For all $g\leq 7$, the Kodaira dimension of $\rr_g$ is $-\infty$.
\end{theorem}

We close by summarizing the structure of the paper. In Section 1 we
introduce the stack $\rer_g$ of Prym curves and determine the Chern
classes of certain tautological vector bundles. In Section 2 we carry out
the enumerative study of the divisors $\overline{\mathcal{D}}_{g:k}$
while in Section 3 we study Koszul divisors on $\rr_g $ in
connection with the Prym-Green Conjecture. The proof of Theorem
\ref{gentype} is completed in Section 4 while Section 5 is concerned
with the enumerative geometry of $\rr_g$ for $g\leq 5$. In Section 6 we describe the behaviour of singularities
of pluricanonical forms of $\rr_g$. There is a significant overlap between some of the results
of this paper and those of \cite{Be}.  Among the things we use from \cite{Be} we mention the description of the branch locus
of $\pi$ and the fact that $\rr_g$ is isomorphic to the coarse
moduli space of $\rem_g(\mathcal{B}\mathbb Z_2)$ (see Section 1). However, some of the results in \cite{Be} are not correct, in particular the statement in \cite{Be} Chapter 3 on singularities of $\rr_g$. Hence we carried out a  detailed study of singularities of $\rr_g$ in Section 6 of our paper.

\section{The stack of Prym curves}

In this section we review a few facts about compactifications of
$\cR_g$. As a matter of terminology, if $\textbf{M}$ is a Deligne-Mumford stack, we denote
by $\cM$ its coarse moduli space (This is contrary to the convention set in \cite{ACV}
but it makes sense, at least from a historical point of view). All the Picard groups of stacks or schemes we are going to consider are with rational coefficients.

We recall that $\pi:\cR_g\rightarrow \cM_g$ is the
$(2^{2g}-1)$-sheeted cover which forgets the point of order $2$ and
we denote by $\rr_g$ the normalization of $\mm_g$ in the function
field of $\cR_g$. By definition, $\rr_g$ is a normal variety and
$\pi$ extends to a finite ramified covering $\pi:\rr_g\rightarrow
\mm_g$. The local behaviour of this branched cover has been studied
in the thesis of M. Bernstein \cite{Be} as well as in the paper
\cite{BCF}. In particular, the scheme $\rr_g$ has two distinct modular
incarnations which we now recall. If $X$ is a nodal curve, a smooth
rational component $E\subset X$ is said to be \emph{exceptional} if
$\#(E\cap \overline{X-E})=2$. The curve $X$ is said to be
\emph{quasi-stable} if any two exceptional components of $X$ are
disjoint. Thus a quasi-stable curve is obtained from a stable curve
by blowing-up each node at most once. We denote by $[st(X)]\in
\mm_g$ the stable model of $X$. We have the following definition
(cf. \cite{BCF}):

\begin{definition}\label{prymstructures} A \emph{Prym curve} of genus
$g$ consists of a triple $(X, \eta, \beta)$, where $X$ is a genus
$g$ quasi-stable curve, $\eta\in \mathrm{Pic}^0(X)$ is a line bundle
of degree $0$ such that $\eta_{E}=\OO_E(1)$ for every exceptional
component $E\subset X$, and $\beta:\eta^{\otimes 2}\rightarrow
\OO_X$ is a sheaf homomorphism which is generically non-zero along
each non-exceptional component of $X$.
\newline
A \emph{family of Prym curves} over a base scheme $S$ consists of a
triple $(\mathcal{X}\stackrel{f}\rightarrow S, \eta, \beta)$, where
$f:\mathcal{X}\rightarrow S$ is a flat family of quasi-stable
curves, $\eta\in \mathrm{Pic}(\mathcal{X})$ is a line bundle and
$\beta:\eta^{\otimes 2}\rightarrow \OO_{\mathcal{X}}$ is a sheaf
homomorphism, such that for every point $s\in S$ the restriction
$(X_s, \eta_{X_s}, \beta_{X_s}:\eta_{X_s}^{\otimes 2}\rightarrow
\OO_{X_s})$ is a Prym curve.
\end{definition}

We denote by $\rer_g$ the non-singular Deligne-Mumford stack of Prym curves of genus $g$. The main result of \cite{BCF} is that the coarse
moduli space of $\rer_g$ is isomorphic to the normalization of
$\mm_g$ in the function field of $\mathcal{R}_g$. On the other hand,
it is proved in \cite{Be} that $\rr_g$ is also isomorphic to the
coarse moduli space of the Deligne-Mumford stack $\rem_g(\mathcal{B}\mathbb Z_2)$ of $\mathbb Z_2$-admissible double
covers introduced in \cite{B} and later in \cite{ACV}. For intersection theory calculations
the language of Prym curves is better suited than that of admissible
covers. In particular, the existence of a degree $0$ line bundle
$\eta$ over the universal Prym curve will be often used to compute
the Chern classes of various tautological vector bundles defined
over $\rer_g$. Throughout this paper we use the isomorphism between rational Picard
groups $\epsilon^*:\mathrm{Pic}(\rr_g)\rightarrow \mathrm{Pic}(\rer_g)$ induced by the map
$\epsilon:\rer_g\rightarrow \rr_g$ from the stack to its coarse moduli space.
\begin{remark} If $(X, \eta, \beta)$ is a Prym curve with exceptional components
$E_1, \ldots, E_r$  and $\{p_i, q_i\}=E_i\cap \overline{X-E_i}$ for
$i=1, \ldots, r$, then obviously $\beta_{E_i}=0$. Moreover,  if
$\tilde{X}:=\overline{X-\bigcup_{i=1}^r E_i}$ (viewed as a subcurve
of $X$), then we have an isomorphism of sheaves
\begin{equation}\label{isom}
\eta^{\otimes 2}_{\tilde{X}}\stackrel{\sim}\rightarrow
\OO_{\tilde{X}}(-p_1-q_1-\cdots -p_r-q_r).
\end{equation}
\end{remark}

It is straightforward to describe all Prym curves $[X, \eta,
\beta]\in \rr_g$ whose stable model has a prescribed topological
type. We do this when $st(X)$ is a $1$-nodal curve and we determine
in the process the boundary components of $\rr_g-\cR_g$.

\begin{example} (\emph{Curves of compact type}) If $st(X)=C\cup D$ is a union of two smooth curves $C$ and
$D$ of genus $i$ and $g-i$ respectively meeting transversally at a point, we use (\ref{isom}) to note
that $X=C\cup D$ (that is, $X$ has no exceptional components).
The line bundle $\eta$ on $X$  is determined by the choice of two line
bundles $\eta_C\in \mathrm{Pic}^0(C)$ and $\eta_D\in
\mathrm{Pic}^0(D)$ satisfying $\eta_C^{\otimes 2}=\OO_C$ and
$\eta_D^{\otimes 2}=\OO_D$ respectively. This shows that for $1\leq
i\leq [g/2]$ the pull-back under $\pi$ of the boundary divisor
$\Delta_i\subset \mm_g$ splits into three irreducible components
$$\pi^*(\Delta_i)=\Delta_i+\Delta_{g-i}+\Delta_{i:g-i},$$ where the
generic point of $\Delta_i\subset \rr_g$ is of the form $[C\cup D,
\eta_C\neq \OO_C, \eta_D=\OO_D]$, the generic point of
$\Delta_{g-i}$ is of the form $[C\cup D, \eta_C=\OO_C, \eta_D\neq
\OO_D])$, and finally $\Delta_{i: g-i}$ is the closure of the locus
of points $[C\cup D, \eta_C\neq \OO_C, \eta_D\neq \OO_D]$ (see
also \cite{Be} pg. 9).
\end{example}

\begin{example} (\emph{Irreducible one-nodal curves}) If $st(X)=C_{yq}:=C/y\sim q$, where $[C, y, q]\in \cM_{g-1, 2}$,
then there are two possibilities, depending on whether $X$ has an
exceptional component or not. Suppose first that $X=C'$ and $\eta
\in \mbox{Pic}^0(X)$. If $\nu:C\rightarrow X$ is the normalization
map, then there is an exact sequence
$$1\longrightarrow \mathbb C^* \longrightarrow
\mbox{Pic}^0(X)\stackrel{\nu^*}\longrightarrow
\mbox{Pic}^0(C)\longrightarrow 0.$$ Thus $\eta$ is determined by a
(non-trivial) line bundle $\eta_C:=\nu^*(\eta)\in \mbox{Pic}^0(C)$
 satisfying $\eta_C^{\otimes 2}=\OO_C$ together with an
identification of the fibres  $\eta_C(y)$ and $\eta_C(q)$. If
$\eta_C=\OO_C$, then there is a unique  way to identify the fibres
$\eta_C(y)$ and $\eta_C(q)$ such that $\eta\neq \OO_X$, and this
corresponds to the classical Wirtinger cover of $X$. We denote by
$\Delta_0^{''}=\Delta_0^{\mathrm{wir}}$ the closure in $\rr_g$ of
the locus of Wirtinger covers. If $\eta_C\neq \OO_C$, then for each
such choice of $\eta_C\in \mbox{Pic}^0(C)[2]$ there are $2$ ways to
glue $\eta_C(y)$ and $\eta_C(q)$. This provides another $2\times
(2^{2g-2}-1)$ Prym curves having $C'$ as their stable model. We set
$\Delta_0'\subset \rr_g$ to be the closure of the locus of  Prym
curves with $\eta_C\neq \OO_C$.

We now treat the case when $X=C\cup_{\{y, q\}} E$, with $E$ being an
exceptional component. Then $\eta_E=\OO_E(1)$ and $\eta_C^{\otimes
2}=\OO_C(-y-q)$. The analysis carried out in \cite{BCF} Proposition
12, shows that $\pi$ is simply ramified at each of these $2^{2g-2}$
Prym curves in $\pi^{-1}([C'])$. We denote by
$\Delta_0^{\mathrm{ram}}\subset \rr_g$ the closure of the locus of
Prym curves $[C\cup_{\{y, q\}} E, \eta, \beta]$ and then
$\Delta_{0}^{\mathrm{ram}}$ is the ramification divisor of $\pi$.
Moreover one has the relation,
$$\pi^*(\Delta_0)=\Delta_0'+\Delta_0^{''}+2\Delta_0^{\mathrm{ram}}.$$
\end{example}
It is easy to establish a dictionary between Prym curves and
Beauville admissible covers. We explain how to do this in
codimension $1$ in $\rr_g$ (see also \cite{D2} Example 1.9). The
general point of $\Delta_0'$ corresponds to an \'etale double cover
$[\tilde{C}\stackrel{f}\rightarrow C]\in \cR_{g-1}$ induced by
$\eta_C$. We denote by $y_i, q_i (i=1, 2)$ the points lying in
$f^{-1}(y)$ and $f^{-1}(q)$ respectively. Then $$\mm_{2g-1}\ni
\frac{\tilde{C}}{y_1\sim q_1, y_2\sim q_2}\longrightarrow
\frac{C}{y\sim q}\in \mm_g$$ is a admissible double cover, defined
up to a sign. This ambiguity is then resolved in the choice of an
element in $\mbox{Ker}\{\nu^*:\mbox{Pic}^0(C_{yq})[2]\rightarrow
\mbox{Pic}^0(C)[2]\}$.

If $[C/y\sim q, \eta, \beta]$ is a general point of $\Delta_0^{''}$,
then we take identical copies $[C_1, y_1, q_1]$ and $[C_2, y_2,
q_2]$ of $[C, y, q]\in \cM_{g-1, 2}$. The Wirtinger cover is
obtained by taking
$$\mm_{2g-1}\ni \frac{C_1\cup C_2}{y_1\sim q_2,
y_2\sim q_1} \longrightarrow \frac{C}{y\sim q}\in \mm_g. $$

If $[C\cup_{\{y, q\}} E, \eta, \beta]\in \Delta_0^{\mathrm{ram}}$,
then $\eta_C\in \sqrt{\OO_C(-y-q)}$ induces a $2:1$ cover
$\tilde{C}\stackrel{f}\rightarrow C$ branched over $y$ and $q$. We
set $\{\tilde{y}\}:= f^{-1}(y), \{\tilde{q}\}:=f^{-1}(q)$. The
Beauville cover is
$$\mm_{2g-1}\ni \frac{\tilde{C}}{\tilde{y}\sim
\tilde{q}}\longrightarrow \frac{C}{y\sim q}\in \mm_g.$$

As usual, one denotes by $\delta_0', \delta_0^{''},
\delta_0^{\mathrm{ram}}, \delta_{i}, \delta_{g-i}, \delta_{i:g-i}\in
\mbox{Pic}(\rer_g)$ the stacky divisor classes corresponding to the
boundary divisors of $\rr_g$. We also set
$\lambda:=\pi^*(\lambda)\in \mbox{Pic}(\rer_g)$. Next we determine
the canonical class $K_{\rr_g}$:

\begin{theorem}
One has the following formula in $\mathrm{Pic}(\rer_g)$:
$$K_{\rr_g}=13\lambda-2(\delta_0^{'}+\delta_0^{''})-3\delta_0^{\mathrm{ram}}-2\sum_{i=1}^{[g/2]}
(\delta_i+\delta_{g-i}+\delta_{i:
g-i})-(\delta_1+\delta_{g-1}+\delta_{1: g-1}).$$
\end{theorem}
\begin{proof} We use that $K_{\mm_g}\equiv 13\lambda-2\delta_0-3\delta_1-2\delta_2-\cdots
-2\delta_{[g/2]}$ (cf. \cite{HM}), together with the Hurwitz formula
for the cover $\pi:\rr_g\rightarrow \mm_g$. We find that
$K_{\rr_g}=\pi^*(K_{\mm_g})+\delta_0^{\mathrm{ram}}$.
\end{proof}

Using this formula as well as the Appendix, we conclude that in
order to prove that $\rr_g$ is of general type for a certain $g$, it
suffices to exhibit a single effective divisor $$D\equiv
a\lambda-b_0'\delta_0'-b_0^{''}\delta_{0}^{''}-b_0^{\mathrm{ram}}\delta_{0}^{\mathrm{ram}}
-\sum_{i=1}^{[g/2]} (b_i\delta_i+b_{g-i}\delta_{g-i}+b_{i:
g-i}\delta_{i: g-i})\ \in \mathrm{Eff}(\rr_g),$$ satisfying the
following inequalities: \begin{equation}\label{inequ}
\mathrm{max}\bigl\{\frac{a}{b_0'},
\frac{a}{b_0^{''}}\bigr\}<\frac{13}{2}, \ \ \mbox{  }
\mathrm{max}\bigl\{\frac{a}{b_0^{\mathrm{ram}}}, \frac{a}{b_1},
\frac{a}{b_{g-1}}, \frac{a}{b_{1: g-1}}\bigr\}<\frac{13}{3}
\end{equation}
and
$$
\mathrm{max}_{i\geq 1}\bigl\{\frac{a}{b_i}, \frac{a}{b_{g-i}},
\frac{a}{b_{i: g-i}}\bigr\}<\frac{13}{2}.
$$

\subsection{The universal Prym
curve}

We start by introducing the partial compactification
$\widetilde{\cM}_g:=\cM_g\cup \widetilde{\Delta}_0$ of $\cM_g$,
obtaining by adding to $\cM_g$ the locus
$\widetilde{\Delta}_0\subset \mm_g$ of one-nodal irreducible curves
$[C_{yq}:=C/y\sim q]$, where $[C, y, q]\in \cM_{g-1, 2}$. Let
$p:\pem_{g, 1}\rightarrow \pem_g$ denote the
universal curve. We  denote
$\widetilde{\cR}_g:=\pi^{-1}(\widetilde{\cM}_g)\subset \rr_g$ and
note that the boundary divisors
$\widetilde{\Delta}_0':=\Delta_0'\cap \pr_g,\
\widetilde{\Delta}_0^{''}:=\Delta_0^{''}\cap \pr_g$ and
$\widetilde{\Delta}_0^{\mathrm{ram}}:=\Delta_0^{\mathrm{ram}}\cap
\pr_g$ become disjoint inside $\pr_g$. Finally, we set
$\cZ:=\per_g\times _{\pem_{g}} \pem_{g, 1}$ and denote by
$p_1:\cZ\rightarrow \per_g$ the projection.

To obtain the universal family of Prym curves over $\per_g$, we
blow-up the codimension $2$ locus $V\subset \cZ$ corresponding to
points
$$v=\bigl([C\cup_{\{y, q\}} E, \eta_C\in \sqrt{\OO_C(-y-q)}], \ \eta_E=\OO_E(1),
\  \nu(y)=\nu(q)\bigr) \in \Delta_0^{\mathrm{ram}}\times _{\pem_
g} \pem_{g, 1}$$ (recall that  $\nu:C\rightarrow C_{yq}$ denotes
the normalization map). Suppose that $(t_1, \ldots, t_{3g-3})$ are
local coordinates in an \'etale neighbourhood of $[C\cup_{\{y, q\}}
E, \eta_C, \eta_E]\in \pr_g$ such that the local equation of
$\Delta_0^{\mathrm{ram}}$ is $(t_1=0)$. Then $\cZ$  around $v$
admits local coordinates $(x, y, t_1, \ldots, t_{3g-3})$ satisfying
the equation $xy=t_1^2$. In particular, $\cZ$ is singular along $V$.
We denote by $\mathcal{X}:=\mbox{Bl}_{V}(\cZ)$ and  by
$f:\mathcal{X} \rightarrow \per_g$ the induced family of Prym curves.
Then for every $[X, \eta, \beta]\in \pr_g$ we have that $f^{-1}([X,
\eta, \beta])=X$.

On $\mathcal{X}$  there exists a  Prym line bundle $\P \in
\mathrm{Pic}(\mathcal{X})$ as well as a morphism of $\OO_X$-modules
$B:\P^{\otimes 2}\rightarrow \OO_{\mathcal{X}}$ with the property
that $\P_{| f^{-1}([X, \eta, \beta])}=\eta$ and $B_{| f^{-1}([X,
\eta, \beta])}=\beta:\eta^{\otimes 2}\rightarrow \OO_X$, for all
points $[X, \eta, \beta]\in \pr_g$ (see e.g. \cite{C}, the same argument carries
over from the spin to the Prym moduli space).

We set $\E_0'$, $\E_0^{''}$ and $\E_0^{\mathrm{ram}}\subset \mathcal{X}$ to
be the proper transforms of the boundary divisors
$p_1^{-1}(\widetilde{\Delta}_0'),
p_1^{-1}(\widetilde{\Delta}_0^{''})$ and
$p_1^{-1}(\widetilde{\Delta}_0^{\mathrm{ram}})$ respectively.
Finally, we define $\E_0$ to be the exceptional divisor of the
blow-up map $\mathcal{X}\rightarrow \cZ$.

We recall that $g:\mathcal{Y}\rightarrow S$ is a family of nodal curves and $L, M$ are line
bundles on $\mathcal{Y}$, then $\langle L, M\rangle\in \mbox{Pic}(S)$ denotes the bilinear \emph{Deligne pairing} of $L$ and $M$.

\begin{proposition}\label{blowup}
If $f:\mathcal{X}\rightarrow \per_g$ is the universal Prym curve and
$\P\in \mathrm{Pic}(\mathcal{X})$ is the corresponding Prym bundle, then
one has the following relations in $\mathrm{Pic}(\per_g)$:
\begin{enumerate}
\item $\langle\omega_f, \P\rangle=0.$
\item $\langle \OO_{\mathcal{X}}(\E_0), \OO_{\mathcal{X}}(\E_0)\rangle=-2\delta_0^{\mathrm{ram}}$.
\item $\langle\OO_{\mathcal{X}}(\P), \OO_{\mathcal{X}}(\P)\rangle=-\delta_0^{\mathrm{ram}}/2.$
\end{enumerate}
\end{proposition}
\begin{proof} The sheaf homomorphism $B:\P^{\otimes 2}\rightarrow
\OO_{\mathcal{X}}$ vanishes (with order $1$) precisely along the
exceptional divisor $\E_0$, hence $[\E_0]=-2c_1(\P)$. Furthermore,
we have the relations
$f^*(\Delta_0^{\mathrm{ram}})=\E_0^{\mathrm{ram}}+\E_0$ and
$f_*([\E_0^{\mathrm{ram}}]\cdot [\E_0])=2\delta_0^{\mathrm{ram}}$
(In the fibre $f^{-1}([C\cup_{\{y, q\}} E, \eta_C])$ the divisors
$\E_0$ and $\E_0^{\mathrm{ram}}$ meet over two points, corresponding
to whether the marked points equals $y$ or $q$. Now (ii) and (iii)
follow simply from the push-pull formula. For (i), it is enough to
show that $\omega_{f | \E_0}$ is the trivial bundle. This follows
because for any point $[X, \eta, \beta]\in \pr_g$ we have that
$\omega_{X}\otimes \OO_E=0$, for any exceptional component $E\subset
X$.
\end{proof}

We now fix $i\geq 1$ and set $\cN_i:=f_*(\omega_f^{\otimes i}\otimes
\P^{\otimes i})$. Since $R^1 f_*(\omega_f^{\otimes i}\otimes
\P^{\otimes i})=0$,  Grauert's theorem implies that $\cN_i$ is a
vector bundle over $\per_g$ of rank $(g-1)(2i-1)$.
\begin{proposition}\label{syzi}
For each integer $i\geq 1$ the following formula in
$\mathrm{Pic}(\per_g)$ holds:
$$c_1(\cN_i)={i\choose
2}(12\lambda-\delta_0'-\delta_0^{''}-2\delta_0^{\mathrm{ram}})+\lambda-\frac{i^2}{4}
\delta_0^{\mathrm{ram}}.$$
\end{proposition}
\begin{proof}
We apply Grothendieck-Riemann-Roch to the universal Prym curve
$f:\mathcal{X}\rightarrow \per_g$:
$$c_1(\cN_i)=f_*\Bigl[\Bigl(1+ic_1(\omega_f\otimes \mathcal{P})
+\frac{i^2c_1^2(\omega_f\otimes
\mathcal{P})}{2}\Bigr)\Bigl(1-\frac{c_1(\omega_f)}{2}+\frac{c_1^2(\omega_f)+[\mathrm{Sing}(f)]}{12}
\Bigr)\Bigr]_2,$$ and then use Proposition \ref{blowup} and
Mumford's formula
$(\kappa_1)_{\per_g}=12\lambda-\delta_0'-\delta_0^{''}-2\delta_0^{\mathrm{ram}}$.
\end{proof}

\subsection{Inequalities between coefficients of divisors on
$\rr_g$}  We use pencils of curves on $K3$ surfaces to establish
certain inequalities between the coefficients of effective divisors
on $\rr_g$. Using $K3$ surfaces we  construct pencils that fill up
the boundary divisors $\Delta_i, \Delta_{g-i}$ and $\Delta_{i: g-i}$
for $1\leq i\leq [g/2]$ when $g\leq 23$. The use of such pencils in
the context of $\mm_g$ has already been demonstrated in \cite{FP}.

We start with a Lefschetz pencil $B\subset \mm_i$ of curves of genus
$i$ lying on a fixed $K3$ surface $S$. The pencil $B$ is induced by
a family $f:\mathrm{Bl}_{i^2}(S)\rightarrow \PP^1$ which has $i^2$
sections corresponding to the base points and we choose one such
section $\sigma$. Using $B$, for each $g\geq i+1$ we create a genus
$g$ pencil $B_i\subset \mm_g$ of stable curves, by gluing a fixed
curve $[C_2, p]\in \cM_{g-i, 1}$ along the section $\sigma$ to each
member of the pencil $B$. Then we have the following formulas on
$\mm_g$ (cf. \cite{FP} Lemma 2.4):
$$B_i\cdot \lambda=i+1,\ B_i\cdot \delta_0=6i+18,\ B_i\cdot
\delta_i=-1\ \mbox{ and } B_i\cdot \delta_j=0\ \mbox{ for } j\neq
i.$$ We fix $1\leq i\leq [g/2]$ and lift $B_i$ in three different
ways to pencils in $\rr_g$. First we choose a non-trivial line
bundle $\eta_{2}\in \mbox{Pic}^0(C_2)[2]$. Let us denote by
$A_{g-i}\subset \Delta_{g-i}\subset \rr_g$ the pencil of Prym curves
$[C_2\cup_{\sigma(\lambda)} f^{-1}(\lambda),\  \eta_{C_2}=\eta_2,\
\eta_{f^{-1}(\lambda)}=\OO_{f^{-1}(\lambda)}]$, with $\lambda\in
\PP^1$.

Next, we denote by $A_i\subset \Delta_i\subset \rr_g$ the pencil
consisting of Prym curves
$$\bigl[C_2\cup_{\sigma(\lambda)} f^{-1}(\lambda), \ \eta_{C_2}=\OO_{C_2},\
\eta_{f^{-1}(\lambda)}\in
\overline{\mathrm{Pic}}^0(f^{-1}(\lambda))[2]\bigr], \ \mbox{ where
}\lambda\in \PP^1.$$

Clearly $\pi(A_i)=B_i$ and $\mbox{deg}(A_i/B_i)=(2^{2i}-1)$.
Finally, $A_{i: g-i}\subset \Delta_{i: g-i}\subset \rr_g$ denotes
the pencil of Prym curves $\bigl[C_2\cup f^{-1}(\lambda),
\eta_{C_2}=\eta_2, \ \eta_{f^{-1}(\lambda)}\in
\overline{\mathrm{Pic}}^0(f^{-1}(\lambda))[2]\bigr].$ Again, we have
that $\mbox{deg}(A_{i: g-i}/B_i)=2^{2i}-1$.

\begin{lemma}\label{pencils}
If $A_i, A_{g-i}$ and $A_{i: g-i}$ are  pencils  defined above, we
have the following relations:
\begin{itemize}
\item $A_{g-i}\cdot \lambda=i+1, \ A_{g-i}\cdot \delta_0'=6i+18, \
A_{g-i}\cdot \delta_i=A_{g-i}\cdot \delta_0^{\mathrm{ram}}=0,\
\mbox{ and } A_{g-i}\cdot \delta_{g-i}=-1.$
\item $A_{i}\cdot \lambda=(i+1)(2^{2i}-1), \ A_i\cdot
\delta_0'=(2^{2i-1}-2)(6i+18), \  \ A_i\cdot \delta_0^{''}=6i+18,
\newline A_i\cdot \delta_0^{\mathrm{ram}}=2^{2i-2}(6i+18)\ \ \mbox{ and } A_i\cdot
\delta_i=-(2^{2i}-1).$
\item $A_{i: g-i}\cdot \lambda=(i+1)(2^{2i}-1), \ A_{i: g-i}\cdot
\delta_0'=(2^{2i-1}-1)(6i+18), \ \newline A_{i: g-i}\cdot
\delta_0^{\mathrm{ram}}=2^{2i-2}(6i+18), \ A_{i: g-i}\cdot
\delta_{0}^{''}=0 \  \mbox{ and } \ A_{i: g-i}\cdot \delta_{i:
g-i}=-(2^{2i}-1).$
\end{itemize}
\end{lemma}
Note that all these intersections are computed on $\rr_g$. The
intersection numbers of $A_i, A_{g-i}$ and $A_{i: g-i}$ with the
generators of $\mbox{Pic}(\rr_g)$  not explicitly mentioned in Lemma
\ref{pencils} are all equal to $0$.
\begin{proof}
We treat in detail only the case of $A_i$ the other cases being
similar. Using \cite{FP} we find that $(A_i\cdot
\lambda)_{\rr_g}=(\pi_*(A_i)\cdot
\lambda)_{\mm_g}=(2^{2i}-1)(B_i\cdot \lambda)_{\mm_g}$. Furthermore,
since $A_i\cap \Delta_{g-i}=A_i\cap \Delta_{i: g-i}=\emptyset$, we
can write the formulas
$$(A_i\cdot \delta_i)_{\rr_g}=\bigl(A_i\cdot
\pi^*(\delta_i)\bigr)_{\rr_g}=(2^{2i}-1) (B_i\cdot
\delta_i)_{\mm_g}.$$ Clearly $(A_i\cdot
\delta_0^{''})_{\rr_g}=(B_i\cdot \delta_0)_{\mm_g}=6i+18$, whereas
the intersection $A_i\cdot \delta_0'$ corresponds to choosing an
element in $\mathrm{Pic}^0(f^{-1}(\lambda))[2]$, where
$f^{-1}(\lambda)$ is a singular member of $B$. There are
$2(2^{2i-2}-1)(6i+18)$ such choices.
\end{proof}

\begin{proposition}\label{ineq}
Let $D\equiv
a\lambda-b_0'\delta_0'-b_0^{''}\delta_0^{''}-b_0^{\mathrm{ram}}
\delta_0^{\mathrm{ram}}-\sum_{i=1}^{[g/2]}
(b_i\delta_i+b_{g-i}\delta_{g-i}+b_{i: g-i}\delta_{i: g-i})\in
\mathrm{Pic}(\rr_g)$ be the closure in $\rr_g$ of an effective
divisor in $\mathcal{R}_g$. Then if $1\leq i\leq
\mathrm{min}\{[g/2], 11\}$, we have the following inequalities:

 \noindent (1)\
$\mbox{  } a(i+1)-b_0' (6i+18)+b_{g-i}\geq 0.$

\noindent  (2)\ $\mbox{ } a(i+1)-b_0^{\mathrm{ram}}
(6i+18)\frac{2^{2i-2}}{2^{2i}-1}- b_0' (6i+18)
\frac{2^{2i-1}-1}{2^{2i}-1}+ b_{i: g-i}\geq 0.$

\noindent (3)\  $  \ \mbox{ } a(i+1)-b_0^{\mathrm{ram}}
(6i+18)\frac{2^{2i-2}}{2^{2i}-1}- b_0'\ (6i+18)
\frac{2^{2i-1}-2}{2^{2i}-1}-b_0^{''}(6i+18)
\frac{1}{2^{2i}-1}+b_{i}\geq 0.$
\end{proposition}
\begin{proof} We use that that in this range the pencils $A_i,
A_{g-i}$ and $A_{i: g-i}$ fill-up the boundary divisors $\Delta_i,
\Delta_{g-i}$ and $\Delta_{i: g-i}$ respectively, hence $A_i\cdot D,
\ A_{g-i}\cdot D,\  A_{i: g-i}\cdot D\geq 0$.
\end{proof}

\noindent {\emph{Proof of Theorem \ref{r7}.}} We lift the Lefschetz
pencil  $B\subset \mm_g$ corresponding to a fixed $K3$ surface, to a pencil $\tilde{B}\subset
\rr_g$ of Prym curves by taking Prym curves $\tilde{B}:=\{[C_\lambda,
\eta_{C_{\lambda}}]\in \rr_g: [C_{\lambda}]\in B,
\eta_{C_{\lambda}}\in \overline{\mbox{Pic}}^0(C_{\lambda})[2]\}$. We
have the following formulas $$\tilde{B}\cdot \lambda =(2^{2g}-1)(g+1), \tilde{B}\cdot
\delta_0'=(2^{2g-1}-2)(6g+18), \ \tilde{B}\cdot \delta_0^{''}=6g+18,
\ \tilde{B}\cdot \delta_0^{\mathrm{ram}}=2^{2g-2}(6g+18).$$ Furthermore,
$\tilde{B}$ is disjoint from all the remaining boundary classes of
$\rr_g$. One now verifies that $\tilde{B}\cdot K_{\rr_g}<0$
precisely when $g\leq 7$. Since $\tilde{B}$ is a covering curve for
$\rr_g$ in the range $g\leq 11, g\neq 10$, we find that
$\kappa(\rr_g)=-\infty$.
\hfill $\Box$

\section{Theta divisors for vector bundles and geometric loci in $\rr_g$}

We present a general method of constructing geometric divisors on
$\rr_g$. For a fixed point $[C, \eta]\in \cR_g$ we shall study  the
relative position of  $\eta\in \mathrm{Pic}^0(C)[2]$ with respect to
certain pluri-theta divisors on $\mathrm{Pic}^0(C)$.

We start by fixing  a smooth curve $C$. If $E\in U_C(r, d)$ is a semistable vector bundle on $C$ of integer
slope $\mu(E):=d/r\in \mathbb Z$, then following Raynaud \cite{R}, we
introduce the determinantal cycle
$$\Theta_E:=\{\eta\in \mathrm{Pic}^{g-\mu-1}(C): H^0(C, E\otimes
\eta)\neq 0\}.$$ Either $\Theta_E=\mbox{Pic}^{g-\mu-1}(C)$, or else,
$\Theta_E$ is a divisor on $\mbox{Pic}^{g-\mu-1}(C)$ and then
$\Theta_E\equiv r\cdot \theta$. In the latter case we say that
$\Theta_E$ is the \emph{theta divisor} of the vector bundle $E$.
Clearly, $\Theta_E$ is a divisor if and only if $H^0(C, E\otimes
\eta)=0$, for a general bundle $\eta\in \mbox{Pic}^{g-\mu-1}(C)$.

Let us now fix a globally generated line bundle  $L\in
\mbox{Pic}^d(C)$ such that $h^0(C, L)=r+1$. The \emph{Lazarsfeld
vector bundle} $M_L$ of $L$  is defined using the exact sequence on
$C$
$$0\longrightarrow M_L\longrightarrow H^0(C, L)\otimes
\OO_C\longrightarrow L\longrightarrow 0$$ (see also \cite{GL},
\cite{L}, \cite{Vo}, \cite{F1}, \cite{FMP} for many applications of
these bundles). It is customary to denote $Q_L:=M_L^{\vee}$, hence
$\mu(Q_L)=d/r$. When $L=K_C$, one writes $Q_C:=Q_{K_C}$. The vector
bundles $Q_L$ (and all its exterior powers) are semistable under
mild genericity assumptions on $C$ (see \cite{L} or \cite{F1}
Proposition 2.1). In the case $\mu(\wedge^i Q_L)=g-1$, when we
expect $\Theta_{\wedge^i Q_L}$ to be a divisor on $\mbox{Pic}^0(C)$,
we may ask whether for a given point $[C, \eta]\in \cR_g$ the
condition $\eta \in \Theta_{\wedge^i Q_L}$ is satisfied or not.
Throughout this section we denote by $\mathfrak G^r_d\rightarrow
\cM_g$ the Deligne-Mumford stack parameterizing pairs $[C, l]$, where $[C]\in \cM_g$
and $l=(L, V)\in G^r_d(C)$ is a linear series of type $\mathfrak g^r_d$.

We fix integers $k\geq 2$ and $b\geq 0$. We set integers
$i:=kb+k-b-2$,
$$r:=kb+k-2,\ g:=k(kb+k-b-2)+1=ik+1\ \mbox{ and } d:=k(kb+k-2).$$
Since $\rho(g, r, d)=0$, a general curve $[C]\in \cM_g$ carries a
finite number of (obviously complete) linear series $l\in G^r_d(C)$. We denote this
number by
$$N:=g!\frac{1!\ 2!\ \cdots r!}{(k-1)!\ \cdots \ (k-1+r)!}=\mathrm{deg}(\mathfrak G^r_d/\cM_g).$$ We
also note that we can write $g=(r+1)(k-1) \ \mbox{ }\mbox{ and }\
d=rk, $ and  moreover, each line bundle $L\in W^r_d(C)$ satisfies
$h^1(C, L)=k-1$. Furthermore, we compute $\mu(\wedge^i Q_L)=ik=g-1$
and then we introduce the following virtual divisor on $\cR_g$:
$$\mathcal{D}_{g: k}:=\{[C, \eta]\in \cR_g: \exists L\in W^r_{d}(C)\
\mbox{ such that } \ h^0(C, \wedge^i Q_L\otimes \eta)\geq 1\}.$$
>From the definition it follows that $\mathcal{D}_{g: k}$ is either
pure of codimension $1$ in $\cR_g$, or else $\mathcal{D}_{g:
k}=\cR_g$. We shall prove that the second possibility does not
occur.

For $[C, \eta]\in \cR_g$ and $L\in W^r_d(C)$ one has the following
exact sequence on $C$
$$0\longrightarrow \wedge^i M_L\otimes K_C\otimes \eta\longrightarrow \wedge^i
H^0(C, L)\otimes K_C\otimes \eta\longrightarrow \wedge^{i-1}
M_L\otimes L\otimes  K_C\otimes \eta\longrightarrow 0,$$ from which,
using Serre duality,  one derives the following equivalences:
$$[C, \eta]\in \mathcal{D}_{g: k}\Leftrightarrow h^1(C, \wedge^i M_L\otimes K_C\otimes
\eta)\geq 1\Leftrightarrow $$ \begin{equation}\label{divstr}
\wedge^i H^0(C, L)\otimes H^0(C, K_C\otimes \eta)\rightarrow H^0(C,
\wedge^{i-1} M_L\otimes L\otimes K_C\otimes \eta) \mbox{ is not an
isomorphism}. \end{equation}
 Note that obviously \ $\mbox{\ rank}
\bigl(\wedge^i H^0(C, L)\otimes H^0(C, K_C\otimes
\eta)\bigr)={r+1\choose i}(g-1)$, while
$$h^0(C, \wedge^{i-1} M_L\otimes L\otimes K_C\otimes \eta)=\chi(C,
\wedge^{i-1} M_L\otimes L\otimes K_C\otimes \eta)=$$ $$={r\choose
i-1}\bigl(-k(i-1)+d+g-1\bigr)={r+1\choose i}(g-1)$$ (use that $M_L$
is a semistable vector bundle and that $\mu(\wedge^{i-1} M_L\otimes
L\otimes K_C\otimes \eta) >2g-1$).

\begin{remark}
As pointed out in the Introduction, an important particular case is
$k=2$, when $i=b, g=2i+1, r=2i, d=4i=2g-2$. Since
$W_{2g-2}^{g-1}(C)=\{K_C\}$, it follows that $[C, \eta]\in
\mathcal{D}_{2i+1, 2}\Leftrightarrow \eta\in \Theta_{\wedge^i
Q_{C}}$. The main result from \cite{FMP} states that for any $[C]\in
\cM_g$ the Raynaud locus $\Theta_{\wedge^i Q_C}$ is a divisor in
$\mbox{Pic}^0(C)$ (that is, $\wedge^i Q_C$ has a theta divisor) and
we have an equality of cycles
\begin{equation}\label{fmp} \Theta_{\wedge^i Q_{C}}=C_i-C_i\subset
\mbox{Pic}^0(C), \end{equation} where the right-hand-side denotes
the \emph{$i$-th difference variety} of $C$, that is, the image of
the difference map
$$\phi:C_i\times C_i \rightarrow \mbox{Pic}^0(C), \mbox{ } \ \phi(D,
E):=\OO_C(D-E).$$

Using Lazarsfeld's filtration argument \cite{L} Lemma 1.4.1, one
finds that for a generic choice of distinct points $x_1, \ldots,
x_{g-2}\in C$, there is an exact sequence
$$0\longrightarrow \oplus_{l=1}^{g-2} \OO_C(x_l)\longrightarrow
Q_C\longrightarrow K_C\otimes \OO_C(-x_1-\cdots
-x_{g-2})\longrightarrow 0,$$ which implies the inclusion
$C_i-C_i\subset \Theta_{\wedge^i Q_C}$. The importance of
(\ref{fmp}) is that it shows that $\Theta_{\wedge^i Q_C}$ is a
divisor on $\mbox{Pic}^0(C)$, that is, $H^0(C, \wedge^i Q_C\otimes
\eta)=0$ for a generic $\eta\in \mbox{Pic}^0(C)$.
\end{remark}
\begin{theorem} For every genus $g=2i+1$ we have the following
identification of cycles on $\cR_g$:
$$\mathcal{D}_{2i+1: 2}:=\{[C, \eta]\in \cR_g: \eta\in C_i-C_i\}.$$
\end{theorem}

Next we prove that $\mathcal{D}_{g: k}$ is an actual divisor on
$\cR_g$ for any $k\geq 2$ and we achieve this by specialization to
the $k$-gonal locus $\cM_{g, k}^1$  in $\cM_g$.
\begin{theorem}\label{transversality}
Fix $k\geq 2, b\geq 1$ and $g, r, d, i$ defined as above. Then
$\mathcal{D}_{g: k}$ is a divisor on $\cR_g$. Precisely, for a
generic $[C, \eta]\in \cR_g$ we have that $H^0(C, \wedge^i
Q_L\otimes \eta)=0$, for every $L\in W^r_d(C)$.
\end{theorem}
\begin{proof}
Since there is a unique irreducible component of $\mathfrak
G^r_d$ mapping
dominantly onto $\cM_g$, to prove that $\mathcal{D}_{g: k}$
is a divisor it suffices to exhibit a single element $[C, L, \eta]\in
\mathfrak G^r_d$ such that  (1) the Petri map
$$\mu_0(C, L): H^0(C, L)\otimes H^0(C, K_C\otimes L^{\vee})\rightarrow
H^0(C, K_C)$$ is an isomorphism, and (2) for each point $\eta\in
\mbox{Pic}^0(C)[2]$, we have that $\eta\notin \Theta_{\wedge^i Q_L}$.

Proposition 2.1.1 from \cite{CM} ensures that for a generic
$k$-gonal curve  $[C, A]\in \mathfrak G^1_k$ of genus $g=(r+1)(k-1)$
one has that $h^0(C, A^{\otimes j})=j+1$ for $1\leq j\leq r+1$. In
particular there is an isomorphism $\mbox{Sym}^j H^0(C, A)\cong
H^0(C, A^{\otimes j})$. Using this and Riemann-Roch, we obtain that
$h^0(C, K_C\otimes A^{\otimes (-j)})=(k-1)(r+1-j)$ for $0\leq j\leq
r+1$. Thus there is a generically injective rational map $\mathfrak
G^1_k\dashrightarrow \mathfrak G^r_d$ given by $[C, A]\mapsto [C, A^{\otimes r}]$
(The use of such a map has been first pointed out to me in a
different context by S. Keel). We claim that $\mathfrak G^1_k$ maps
into the "main component" of $\mathfrak G^r_d$ which maps dominantly
onto $\mm_g$. To prove this it suffices to check that the Petri map
$$\mu_0(C, A^{\otimes r}): H^0(C, A^{\otimes r})\otimes H^0(C, K_C\otimes A^{\otimes
(-r)})\rightarrow H^0(C, K_C)$$ is an isomorphism (Remember that
$H^0(C, A^{\otimes r})\cong \mbox{Sym}^r H^0(C, A)$). We use the
base point free pencil trick to write down the exact sequence
$$0\longrightarrow H^0(K_C\otimes A^{\otimes -(j+1)})\longrightarrow H^0(
A)\otimes H^0(K_C\otimes A^{\otimes
(-j)})\stackrel{\mu_j(A)}\longrightarrow H^0(K_C\otimes A^{\otimes
-(j-1)}).$$ One can now easily check that  $\mu_j(A)$ is surjective
for $1\leq j\leq r$ by using the formulas $h^0(C, K_C\otimes
A^{\otimes (-j)})=(k-1)(r+1-j)$ valid for $0\leq j\leq r+1$. This in
turns implies that $\mu_0(C, A^{\otimes r})$ is surjective, hence an
isomorphism.

\vskip 3pt

We now check condition (2) and note that for $[C, L=A^{\otimes
r}]\in \mathfrak G^r_d$, the Lazarsfeld bundle splits as $Q_L\cong
A^{\oplus r}$. In particular, $\wedge^i Q_L\cong \oplus_{{r\choose
i}} A^{\otimes i}$, hence the condition $H^0(C, \wedge^i Q_L\otimes
\eta)\neq 0$ is equivalent to $H^0(C, A^{\otimes i}\otimes \eta)\neq
0$, that is, the translate of the theta divisor
$W_{g-1}(C)-A^{\otimes i}\subset \mbox{Pic}^0(C)$ cannot contain any
point of order $2$ on $\mbox{Pic}^0(C)$. Using that the moduli space of triples $[C,A,\eta]$, where 
$[C,A]\in \mathfrak{G}^1_k$ and $\eta\in \mbox{Pic}^0(C)[2]$ is irreducible for each $k\geq 3$, it suffices to prove the statement for a single such triple.

\vskip 3pt

We assume by contradiction
that for \emph{any} $[C, A]\in \mathfrak G^1_k$ and \emph{any}
$\eta\in \mbox{Pic}^0(C)[2]$, we have that $H^0(C, A^{\otimes
i}\otimes \eta)\geq 1$. We specialize $C$ to a hyperelliptic curve and choose $A=\mathfrak
g^1_2\otimes \OO_C(x_1+\cdots+x_{k-2})$, with $x_1, \ldots,
x_{k-2}\in C$ being general points. Finally we take
$\eta:=\OO_C(p_1+\cdots+p_{i+1}-q_{1}-\cdots-q_{i+1})\in
\mbox{Pic}^0(C)[2]$, with $p_1, \ldots, p_{i+1}, q_1, \ldots,
q_{i+1}$ being distinct ramification points of the hyperelliptic
$\mathfrak g^1_2$. It is now straightforward to check that $H^0(C,
A^{\otimes i}\otimes \eta)=0$.
\end{proof}

In order to compute the class $[\overline{\mathcal{D}}_{g: k}]\in
\mbox{Pic}(\rr_g)$ we extend the determinantal description of
$\mathcal{D}_{g: k}$ to the boundary of $\rr_g$. We start by setting
some notation. We denote by $\textbf{M}_g^0\subset \textbf{M}_g$ the open substack
classifying curves $[C]\in \cM_g$ such that $W_{d-1}^r(C)=
\emptyset$ and $W_d^{r+1}(C)= \emptyset$. We know that
$\mbox{codim}(\cM_g-\cM_g^0, \cM_g)\geq 2$. We further denote by
$\Delta_0^0\subset \Delta_0\subset \mm_g$ the locus of curves
$[C/y\sim q]$ where $[C]\in \cM_{g-1}$ is a curve that satisfies the
Brill-Noether theorem and where $y, q\in C$ are arbitrary points. Note that every Brill-Noether general curve $[C]\in \cM_{g-1}$ satisfies  $$W_{d-1}^{r}(C)=\emptyset,  \
\ W_d^{r+1}(C)=\emptyset\  \mbox{ and } \mbox{ dim }W^r_d(C)=\rho(g-1, r,
d)=r.$$ We set $\rem_g^0:=\textbf{M}_g^0\cup \Delta_0^0\subset \rem_g$. Then
we consider the  Deligne-Mumford stack
$$\sigma_0:\mathfrak G^r_d\rightarrow \rem_g^0$$
classifying pairs $[C, L]$ with $[C]\in \mm_g^0$ and $L\in G^r_d(C)$ (cf. \cite{EH}, \cite{F2}, \cite{Kh} -note that it is essential that $\rho(g, r, d)=0$. At the moment there is no known extension of this stack over the entire $\rem_g$).
We remark that for any curve $[C]\in \mm_{g}^0$ and $L\in W^r_d(C)$ we have that
$h^0(C, L)=r+1$, that is, $\mathfrak{G}^r_d$ parameterizes only complete linear series. Indeed, for a smooth curve $[C]\in \cM_{g}^0$ we have that $W_{d}^{r+1}(C)=\emptyset$,
so necessarily $W^r_d(C)=G^r_d(C)$.
For a point $[C_{yq}:=C/y\sim q]\in \Delta_0^0$ we have the identification
$$\sigma_0^{-1}\bigl[C_{yq}\bigr]=\{L\in W^r_d(C): h^0(C,
L\otimes \OO_C(-y-q))=r\},$$
where we note that since the normalization $[C]\in \cM_{g-1}$ is assumed to be Brill-Noether general, any sheaf  $L\in \sigma_0^{-1}[C_{yq}]$ satisfies $h^0(C, L\otimes \OO_C(-y))=h^0(C, L\otimes \OO_C(-q))=r$ and $h^0(C, L)=r+1$. Furthermore,  $\sigma_0:\mathfrak
G^r_d\rightarrow \rem_g^0$ is proper, which is to say that
$\overline{W}^r_d(C_{yq})=W^r_d(C_{yq})$, where the left-hand-side
denotes the closure of $W^r_d(C_{yq})$ in the variety
$\overline{\mbox{Pic}}^d(C_{yq})$ of torsion-free sheaves on
$C_{yq}$. This follows because a non-locally free torsion-free sheaf in
$\overline{W}^r_d(C_{yq})-W^r_d(C_{yq})$ is of the form $\nu_*(A)$,
where $A\in W_{d-1}^r(C)$ and $\nu:C\rightarrow C_{yq}$ is the
normalization map. But we know that $W_{d-1}^r(C)=\emptyset$,
because $[C]\in \cM_{g-1}$ satisfies the Brill-Noether
theorem. Since $\rho(g, r, d)=0$, by general Brill-Noether theory,
there exists a unique irreducible component of $\mathfrak G^r_d$
which maps onto $\rem_g^0$. It is certainly not the case that
$\mathfrak G^r_d$ is irreducible, unless $k\leq 3$, when either
$\mathfrak G^r_d=\textbf{M}_g$ ($k=2$), or $\mathfrak G^r_d$ is isomorphic
to a Hurwitz stack ($k=3$). We denote by $f^r_d:\mathfrak C_{g,
d}^r:=\rem_{g, 1}^0\times_{\rem_g^0} \mathfrak G^r_d\rightarrow
\mathfrak G^r_d$ the pull-back of the universal curve $\rem_{g,
1}^0\rightarrow \rem_g^0$ to $\mathfrak G^r_{d}$. Once we have chosen a
Poincar\'e bundle $\L$ on $\mathfrak C^r_{g, d}$ we can form the
three codimension $1$ tautological classes in $A^1(\mathfrak
G^r_d)$:
\begin{equation}\label{tautological}
\mathfrak{a}:=(f^r_d)_*\bigl(c_1(\L)^2\bigr), \ \mathfrak{b}:=(f^r_d)_*\bigl(c_1(\L)\cdot
c_1(\omega_{f_d^r})\bigr), \mbox{ }
\mathfrak{c}:=(f^r_d)_*\bigl(c_1(\omega_{f_d^r})^2\bigr)=(\sigma_0)^*\bigl((\kappa_1)_{
\rem_g^0}\bigr).
\end{equation} These classes depend on the choice of
$\L$ and behave functorially with respect to base change, see also Remark \ref{poinc}
on the precise statement regarding the choice of $\L$. We set
$\rer_g^0:=\pi^{-1}(\pem_g^0)\subset \per_g$ and introduce
the stack of $\mathfrak g^r_d$'s on Prym curves
$$\sigma:\mathfrak G^r_d(\per_g^0/\pem_g^0):=\rer_g^{0} \times_{\rem_g^0} \mathfrak
G^r_d\rightarrow \rer_g^0 .$$ By a slight abuse of notation we denote
the boundary divisors by the same symbols, that is,
$\Delta_0':=\sigma^*(\Delta_0'),
\Delta_0^{''}:=\sigma^*(\Delta_0^{''})$ and
$\Delta_0^{\mathrm{ram}}:=\sigma^*(\Delta_0^{\mathrm{ram}})$.
Finally, we introduce the universal curve over the stack of
$\mathfrak g^r_d$'s on Prym curves:
$$f':\mathcal{X}_d^r:=\mathcal{X}\times _{\rer_g^0} \mathfrak
G^r_d(\rer_g^0/\rem_g^0)\rightarrow \mathfrak
G^r_d(\rer_g^0/\rem_g^0).$$ On $\mathcal{X}^r_d$ there are two
tautological line bundles, the universal Prym bundle
$\mathcal{P}_d^r$ which is the pull-back of $\mathcal{P}\in
\mbox{Pic}(\mathcal{X})$ under the projection
$\mathcal{X}_d^r\rightarrow \mathcal{X}$, and a  Poincar\'e bundle
$\mathcal{L}\in \mbox{Pic}(\mathcal{X}_d^r)$ characterized by the
property $\mathcal{L}_{|f'^{-1}[X, \eta, \beta, L]}=L\in W^r_d(C),$
for each point $[X, \eta, \beta, L]\in \mathfrak
G^r_d(\rr_g^0/\mm_g^0)$. Note that we also have the codimension $1$
classes $\mathfrak{a}, \mathfrak{b}, \mathfrak{c}\in A^1(\AUX)$ defined by the formulas
(\ref{tautological}).

\begin{proposition}\label{lazvanish}
Let $C$ be a curve of genus $g$ and let $L\in W^r_d(C)$ be a
globally generated complete linear series. Then for any integer
$0\leq j\leq r$ and for any line bundle $A\in \mathrm{Pic}^a(C)$
such that $a\geq 2g+d-r+j-1$, we have that $H^1(C, \wedge^j
M_L\otimes A)=0$.
\end{proposition}

\begin{proof} We use a filtration argument due to Lazarsfeld \cite{L}.
Having fixed $L$ and $A$, we choose general points $x_1,\ldots,
x_{r-1}\in C$ such that $h^0\bigl(C, L\otimes
\OO_C(-x_1-\cdots-x_{r-1})\bigr)=2$ and then there is an exact
sequence on $C$
$$0\longrightarrow L^{\vee}(x_1+\cdots+x_{r-1})\longrightarrow
M_L\longrightarrow \oplus_{l=1}^{r-1} \OO_C(-x_l)\longrightarrow
0.$$ Taking the $j$-th exterior powers and tensoring the resulting
exact sequence with $A$, we find that in order to conclude that
$H^1(C, \wedge^i M_L\otimes A)=0$ for $i\leq r$, it suffices to show
that for $1\leq i\leq r$ the following hold:

\noindent (1) $H^1\bigl(C, A\otimes \OO_C(-D_j)\bigr)=0$ for each
effective divisor $D_j\in C_j$ with support in the set $\{x_1,
\ldots, x_{r-1}\}$, and

\noindent (2) $H^1\bigl(C, A\otimes L^{\vee}\otimes
\OO_C(D_{r-j})\bigr)=0$, for any effective divisor $D_{r-j}\in
C_{r-j}$ with support contained in $\{x_1, \ldots, x_{r-1}\}$.

Both (1) and (2) hold for degree reasons since  $\mbox{deg}(C,
A\otimes \OO_C(-D_j))\geq 2g-1$ and $\mbox{deg}(C, A\otimes
L^{\vee}\otimes \OO_C(D_{r-j})\geq 2g-1$ and the points $x_1,
\ldots, x_{r-1}\in C$ are general.
\end{proof}

Next we use Proposition \ref{lazvanish} to prove a vanishing result
for Prym curves.
\begin{proposition}\label{vanish}
For each point $[X, \eta, \beta, L]\in \mathfrak
G^r_d(\rer_g^0/\rem_g^0)$ and $0\leq a\leq i-1$, we have that
$$H^1(X, \wedge^a M_L\otimes
L^{\otimes (i-a)}\otimes \omega_X\otimes \eta)=0.$$
\end{proposition}
\begin{proof}
If $X$ is smooth, then the vanishing follows directly from
Proposition \ref{lazvanish}. Assume now that $[X, \eta, \beta]\in
\Delta_0'\cup \Delta_0^{''}$, that is, $st(X)=X$ and $\eta \in
\mbox{Pic}^0(X)[2]$. As usual, we denote by $\nu:C\rightarrow X$ the
normalization map, and $L_C:=\nu^*(L)\in W^r_d(C)$ satisfies $h^0(C,
L_C\otimes \OO_C(-y-q))=r$, hence $H^0(X, L)\cong H^0(C, L_C)$,
which implies that $\nu^*(M_{L})=M_{L_C}$. Tensoring  the usual
exact sequence on $X$
$$0\longrightarrow \OO_X\longrightarrow \nu_*\OO_C\longrightarrow
\nu_*\OO_C/\OO_X\longrightarrow 0,$$  by the line bundle $\wedge^a
M_L\otimes L^{(i-a)}\otimes \omega_X\otimes \eta$,  we find that a
sufficient condition for the vanishing $H^1(X, \wedge^a M_L\otimes
L^{\otimes (i-a)}\otimes \omega_X\otimes \eta)=0$ to hold, is to
show that
$$ H^1(C, \wedge^a M_{L_C}\otimes L_C^{\otimes (i-a)}\otimes
K_C\otimes \eta_C)=H^1(C, \wedge^a M_{L_C}\otimes L_C^{\otimes
(i-a)}\otimes K_C(y+q)\otimes \eta_C)=0.$$ Since $i<r$, this follows
directly from Proposition \ref{lazvanish}.

We are left with the case when $[X, \eta, \beta]\in
\Delta_0^{\mathrm{ram}}$, when $X:=C\cup_{\{q, y\}} E$, with $E$
being a smooth rational curve,  $L_C\in W^r_d(C), L_E=\OO_E$ and
$\eta_C^{\otimes 2}=\OO_C(-y-q)$. We also have that $M_{L |
C}=M_{L_C}$ and $M_{L | E}=H^0(C, L_C\otimes \OO_C(-y-q))\otimes
\OO_E$. A standard argument involving the Mayer-Vietoris sequence on
$X$ shows that the vanishing of the group $H^1(X, \wedge^a
M_L\otimes L^{\otimes (i-a)}\otimes \omega_X\otimes \eta)$ is
implied by the following vanishing conditions $$H^1(C, \wedge^a
M_{L_C}\otimes L_C^{\otimes (i-a)}\otimes K_C(y+q)\otimes
\eta_C)=H^1(C, \wedge^a M_{L_C}\otimes L_C^{\otimes (i-a)}\otimes
K_C\otimes \eta_C)=0.$$ The conditions of Proposition
\ref{lazvanish} being satisfied $(i\leq r-1)$, we finish the proof.
\end{proof}

Proposition \ref{vanish} enables us to define a sequence of
tautological vector bundles on $\AUX$: First, we set $\H:=f'_*(\L)$.
By Grauert's theorem, it follows that $\H$ is a vector bundle of
rank $r+1$ with fibre $\H[X, \eta, \beta, L]=H^0(X, L)$. For $j\geq
0$ we set $$\cA_{0, j}:=f'_*(\L^{\otimes j}\otimes
\omega_{f'}\otimes \P_d^r).$$ Since $R^1f'_*(\L^{\otimes j}\otimes
\omega_{f'}\otimes \P_d^r)=0$ we find that $\cA_{0, j}$ is a vector
bundle over $\AUX$ of rank equal to $h^0(X, L^{\otimes j}\otimes
\omega_X\otimes \eta)=jd+g-1.$  Next we introduce the global
Lazarsfeld vector bundle $\cM$ over $\mathcal{X}_d^r$ by the exact
sequence
$$0\longrightarrow \cM\longrightarrow f'^*(\H)\longrightarrow
\L\longrightarrow 0,$$ hence $\cM_{f'^{-1}[X, \eta, \beta, L]}=M_L$
for each $[X, \eta, \beta, L]\in \AUX$. Then for integers $a, j\geq
1$ we define the sheaf
$$\cA_{a, j}:=f'_*(\wedge^a \cM\otimes \L^{\otimes j}\otimes
\omega_{f'}\otimes \P^r_d).$$ For each $1\leq a\leq i-1$, we have
proved that $R^1 f'_*(\wedge^a \cM\otimes \L^{\otimes (i-a)}\otimes
\omega_{f'}\otimes \P^r_d)=0$ (cf. Proposition \ref{vanish}),
therefore $\cA_{a, i-a}$ is a vector bundle over $\AUX$ having rank
$$\mathrm{rk}(\cA_{a, i-a})=\chi\bigl(X, \wedge^a M_L\otimes L^{\otimes
(i-a)}\otimes \omega_X\otimes \eta\bigr)={r\choose a}k(i-a)(r+1).$$
Proposition \ref{vanish} also shows that for all integers $1\leq a\leq i-1$,
the vector bundles $\cA_{a, i-a}$ sit in exact sequences
\begin{equation}\label{recursion}
0\longrightarrow \cA_{a, i-a}\longrightarrow \wedge^a \H\otimes
\cA_{0, i-a}\longrightarrow \cA_{a-1, i-a+1}\longrightarrow 0.
\end{equation}

We shall need the expression for the Chern numbers of $\cA_{a,
i-a}$. Using (\ref{recursion}) it will be sufficient to compute
$c_1(\cA_{0, j})$ for all $j\geq 0$.
\begin{proposition}\label{rr}
For all $j\geq 0$ one has the following formula in $A^1(\AUX)$:
$$c_1(\cA_{0, j})=\lambda+\frac{j}{2}B+\frac{j^2}{2} A-\frac{1}{4}\delta_0^{\mathrm{ram}}.$$
\end{proposition}
\begin{proof}
We apply Grothendieck-Riemann-Roch to the morphism
$f':\mathcal{X}_d^r\rightarrow \AUX$:
$$c_1(\cA_{0, j})=c_1\bigl(f'_{!}(\omega_{f'}\otimes \L^{\otimes j}\otimes \P^r_d)\bigr)=
$$
$$=f'_*\Bigl[\Bigl(1+c_1(\omega_{f'}\otimes \L^{\otimes j}\otimes
\P^r_d)+\frac{c_1^2(\omega_{f'}\otimes \L^{\otimes j} \otimes
\P^r_d)}{2}\Bigr)\Bigl(1-\frac{c_1(\omega_{f'})}{2}+\frac{c_1^2(\omega_{f'})
+[\mathrm{Sing}(f')]}{12}\Bigr)\Bigr]_2,$$ where
$\mathrm{Sing}(f')\subset \mathcal{X}_{d}^r$ denotes the codimension
$2$ singular locus of the morphism $f'$, therefore
$f'_*[\mathrm{Sing}(f')]=\Delta_0'+\Delta_0^{''}+2\Delta_0^{\mathrm{ram}}$.
We finish the proof using  Mumford's formula
$\kappa_1=f'_*(c_1^2(\omega_{f'}))=12\lambda-(\delta_0'+\delta_0^{''}+
2\delta_0^{\mathrm{ram}})$ and noting that $f'_*(c_1(\L)\cdot
c_1(\mathcal{P}_d^r))=0$ (the restriction of $\L$ to the exceptional
divisor of $f':\mathcal{X}_d^r\rightarrow \AUX$ is trivial)
and
$f'_*(c_1(\omega_{f'})\cdot c_1(\mathcal{P}_d^r))=0$,. Finally,
according to Proposition \ref{blowup} we have that
$f'_*(c_1^2(\mathcal{P}_d^r))=-\delta_0^{\mathrm{ram}}/2$.
\end{proof}
\begin{remark}\label{poinc}
While the construction of the vector bundles $\cA_{a, j}$
 depends on the choice of the Poincar\'e bundle
$\mathcal{L}$ and that of the Prym bundle $\mathcal{P}_d^r$,  it is easy to check that if we set
the vector bundles
$\cA:=\wedge^i \mathcal{H}\otimes \cA_{0, 0}$ and $\cB:=\cA_{i-1, i}$, then
the vector bundle $Hom(\cA, \cB)$ on $\AUX$,  as well as the morphism
$$\phi \in H^0\bigl(\AUX, Hom(\cA, \cB)\bigr)$$ whose degeneracy locus is the virtual divisor
$\overline{\mathcal{D}}_{g: k}$, are independent of
such choices. More precisely, let us denote by $\Xi$ the collection
of triples $\alpha:=\bigl(\pi_{\alpha}, \mathcal{L}_{\alpha}, (\mathcal{P}_d^r)_{\alpha}\bigr)$, where
$\pi_{\alpha} :\Sigma_{\alpha}\rightarrow \AUX$ is an
\'etale surjective morphism from a scheme $\Sigma_{\alpha}$, $(\mathcal{P}_d^r)_{\alpha}$ is
a Prym bundle  and
$\mathcal{L}_{\alpha}$ is a Poincar\'e bundle on $p_{2, \alpha}:
\mathcal{X}_d^r\times_{\AUX} \Sigma_{\alpha}\rightarrow
\Sigma_{\alpha}$. Recall that if $\Sigma \rightarrow
\AUX$ is an \'etale surjection from a scheme and $\mathcal{L}$ and
$\mathcal{L}'$ are two Poincar\'e bundles on $p_2: \mathcal{X}_d^r \times_{\AUX}\Sigma \rightarrow \Sigma$, then the sheaf
$\mathcal{N}:=p_{2 *} Hom (\mathcal{L}, \mathcal{L'})$ is invertible
and there is a canonical isomorphism $\mathcal{L}\otimes
p_2^*\mathcal{N}\cong \mathcal{L}'$. For every $\alpha\in \Xi$ we
construct the morphism between vector bundles of the same rank
$\phi_{\alpha}:\cA_{\alpha}\rightarrow \cB_{\alpha}$ as above. Then
since a straightforward cocycle condition is met, we find that there
exists a vector bundle $Hom(\cA, \cB)$ on $\AUX$ together
with a section $\phi \in H^0(\AUX, Hom (\cA, \cB))$ such
that for every $\alpha=(\pi_{\alpha}, \mathcal{L}_{\alpha}, (\mathcal{P}_d^r)_{\alpha})\in \Xi$
we have that $$\pi_{\alpha}^* (Hom(\cA, \cB))=Hom(\cA_{\alpha},
\cB_{\alpha})\ \mbox{ and }\ \pi_{\alpha}^*(\phi)=\phi_{\alpha}.$$
\end{remark}

We are finally in a position to compute the class of the divisor $\overline{\mathcal{D}}_{g: k}$.

\begin{theorem}\label{det}
We fix integers $k\geq 2, b\geq 0$ and set 
$$i:=kb-b+k-2,\
r:=kb+k-2,\ g:=ik+1, d:=rk$$ as above. Then there exists a morphism
$\phi:\wedge^i \H\otimes \cA_{0, 0}\rightarrow \cA_{i-1, 1}$ between
vector bundles of the same rank over $\AUX$, such that the
push-forward under $\sigma$ of the restriction to $\mathfrak
G^r_d(\textbf{R}^0_g/\textbf{M}^0_g)$ of the degeneration locus of $\phi$ is precisely
the effective divisor $\cD_{g:k}$. Moreover we have the following
expression for its class in $A^1(\rer_g^0)$:
$$\sigma_*\bigl(c_1(\cA_{i-1, 1}-\wedge^i \H\otimes \cA_{0,
0})\bigr)\equiv {r\choose
b}\frac{N}{(r+k)(kr+k-r-3)}\Bigl(\mathfrak{A}\lambda-\frac{\mathfrak{B}_0}{6}(\delta_0'+\delta_0^{''})-
\frac{\mathfrak{B}_0^{\mathrm{ram}}}{12}\delta_0^{\mathrm{ram}}
\Bigr),$$ where
$$\mathfrak{A}=(k^5-4k^4+5k^3-2k^2)b^3+(3k^5-13k^4+24k^3-23k^2+9k)b^2+$$
$$+(3k^5-14k^4+34k^3-45k^2+24k-4)b+k^5-5k^4+15k^3-25k^2+16k-2,$$
$$\mathfrak{B}_0=(k^5-4k^4+5k^3-2k^2)b^3+(3k^5-13k^4+22k^3-17k^2+5k)b^2+$$
$$+(3k^5-14k^4+30k^3-33k^2+14k-2)b+k^5-5k^4+13k^3-19k^2+10k$$
and
$$\mathfrak{B}_0^{\mathrm{ram}}=(4k^5-16k^4+20k^3-8k^2)b^3+(12k^5-52k^4+85k^3-65k^2+20k)b^2+$$
$$+(12k^5-56k^4+111k^3-114k^2+53k-8)b+4k^5-20k^4+46k^3-58k^2+34k-6.$$
\end{theorem}

\begin{proof} To compute the class of the degeneracy locus of $\phi$ we use
the exact sequence (\ref{recursion}) and Proposition \ref{rr}. We
write the following identities in $A^1(\AUX)$:
$$c_1\bigl(\cA_{i-1, 1}-\wedge^i \H\otimes \cA_{0,
0}\bigr)=\sum_{l=0}^i (-1)^{l-1} c_1(\wedge^{i-l} \H\otimes \cA_{0,
l})=$$
$$=\sum_{l=0}^i (-1)^{l-1} \Bigl((ld+g-1){r\choose
i-l-1}c_1(\H)+{r+1\choose i-l}c_1(\cA_{0, l})\Bigr)=$$
$$=-k{kb+k-4\choose b-1} c_1(\H)+\frac{1}{2}{kb+k-3\choose
b}\ \mathfrak{b}-$$ $$-{kb+k-2\choose b}
\lambda-\frac{kb+k-2b-3}{2(kb+k-3)}{kb+k-3 \choose b}\
\mathfrak{a}+\frac{1}{4}{kb+k-2\choose b} \delta_0^{\mathrm{ram}}=
$$
$$={r-1\choose
b}\Bigl(-\frac{kb}{r-1}\ c_1(\H)+\frac{1}{2}\
\mathfrak{b}-\frac{r-2b-1}{2(r-1)}\ \mathfrak{a}-\frac{r}{r-b}\
\lambda+\frac{r}{4(r-b)}\delta_0^{\mathrm{ram}}\Bigr),$$ where
$\delta_0^{\mathrm{ram}}=\sigma^*(\delta_0^{\mathrm{ram}})\in
A^1(\AUX)$. The classes $\mathfrak{a}, \mathfrak{b}\in A^1(\AUX)$ and the vector bundle
$\H$ on $\AUX$ are defined in terms of a
Poincar\'e bundle $\L$: If $\L':=\L\otimes f'^*(\cM)$ is another
Poincar\'e bundle with $\cM\in \mbox{Pic}\bigl(\AUX\bigr)$ and if
$\mathfrak{a}', \mathfrak{b}', \H'$ denote the classes defined in terms of $\L'$ using
(\ref{tautological}), then we have formulas:
$$\mathfrak{a}'=\mathfrak{a}+2dc_1(\cM),\ \mathfrak{b}'=\mathfrak{b}+(2g-2)c_1(\cM) \mbox{ and }
c_1(\H')=c_1(\H)+(r+1)c_1(\cM).$$ A straightforward calculation
shows that the class
\begin{equation}\label{tilxi}
\Xi:=-\frac{kb}{r-1}\ c_1(\H)+\frac{1}{2}\ \mathfrak{b}-\frac{r-2b-1}{2(r-1)}\
\mathfrak{a}\in A^1(\AUX)
\end{equation}
is independent of the choice of $\L$ and
$\sigma_*(\Xi)=\pi^*\bigl((\sigma_0)_*(\Xi_0)\bigr)$, where the
$\Xi_0 \in A^1(\mathfrak G^r_d)$ is defined by the same formula
(\ref{tilxi}) but inside $\mbox{Pic}(\mathfrak G^r_d)$. We outline
below the computation of $\pi^*\bigl((\sigma_0)_*(\Xi_0)\bigr)$
which uses \cite{F2} in an essential way.

We follow closely \cite{F2} and denote by $\rem_g^1:=\textbf{M}_g^0\cup
\Delta_0^0\cup \Delta_1^0$ the partial compactification of $\textbf{M}_g^0$
obtained from $\rem_{g}^0$ by adding the stack $\Delta_1^0\subset
\Delta_1$ consisting of curves $[C\cup_y E]$, where $[C, y]\in
\cM_{g-1, 1}$ is a Brill-Noether general pointed curve and $[E,
y]\in \mm_{1, 1}$. We extend $\sigma_0:\mathfrak G^r_d\rightarrow
\rem_g^0$ to a proper map $\sigma_1:\widetilde{\mathfrak
G}^r_d\rightarrow \rem_{g}^1$ from the Deligne-Mumford stack of limit linear series
$\mathfrak g^r_d$ (cf. \cite{EH}, \cite{F2}, \cite{Kh}). Then for each $n\geq 1$ we consider the vector
bundles $\G_{0, n}$ over $\widetilde{\mathfrak G}^r_d$ defined in
\cite{F2} Proposition 2.8 and which has the following description of
its fibres:
\begin{itemize} \item $\G_{0, n}(C, L)=H^0(C,
L^{\otimes n})$, for each $[C]\in \cM_g^0$ and $L\in W^r_d(C)$.
\item $\G_{0, n}(t)=H^0\bigl(C, L^{\otimes n}(-y-q)\bigr)\oplus \mathbb C\cdot
u^n\subset H^0(C, L^{\otimes n}),$ where the point $t=\bigl(C_{yq},
\ \ L\in W^r_d(C)\bigr)\in \sigma_0^{-1}([C_{yq}])$, with $u\in
H^0(C, L)$ being a section such that $$H^0(C, L)=H^0(C,
L(-y-q))\oplus \mathbb C\cdot u.$$
\item $\G_{0, n}(t)=H^0(C,
L^{\otimes n}(-2y))\oplus \mathbb C\cdot u^n \subset H^0(C,
L^{\otimes n})$, where $t=\bigl(C\cup _y E, l_C, l_E\bigr)\in
\sigma_0^{-1}([C\cup_y E])$ and  $(l_C, l_E)\in G^r_d(C)\times
G^r_d(E)$ being  a limit linear series $\mathfrak g^r_d$ with $l_C=(L, H^0(C, L))$ and
$u\in H^0(C, L)$ a section such that $$H^0(C, L)=H^0(C, L(-2y))\oplus
\mathbb C\cdot u.$$
\end{itemize}
We extend the classes $\mathfrak{a}, \mathfrak{b}\in A^1(\mathfrak G^r_d)$ over the stack
$\widetilde{\mathfrak G}^r_d$ by choosing a Poincar\'e bundle over
$\rem_{g, 1}^1\times _{\rem_{g}^1} \widetilde{\mathfrak G}^r_d$ which
restricts to line bundles of bidegree $(d, 0)$ on curves $[C\cup_y
E]\in \Delta_1^0$. Grothendieck-Riemann-Roch applied to the
universal curve over $\widetilde{\mathfrak G}^r_d$ gives that
\begin{equation}\label{auxi}
c_1(\G_{0, n})=\lambda-\frac{n}{2} \mathfrak{b}+\frac{n^2}{2} \mathfrak{a} \in
A^1(\widetilde{\mathfrak G}^r_d), \ \mbox{ for all } n\geq
2\end{equation} while obviously $\sigma^*(\G_{0, 1})=\H$. We now fix
a general pointed curve $[C, q]\in \cM_{g-1}$ and an elliptic curve
$[E, y]\in \cM_{1, 1}$ and consider the test curves (see also
\cite{F2} p. 7)
$$C^0:=\{C/y\sim q\}_{y\in C}\subset \Delta_0^0\subset \mm_g^1
\ \mbox{ and } \  C^1:=\{C\cup_y E\}_{y\in C}\subset
\Delta_1^0\subset \mm_g^1.$$ For $n\geq 1$, the intersection numbers
$C^0\cdot (\sigma_0)_*(c_1(\G_{0, n}))$ and $C^1\cdot
(\sigma_0)_*(c_1(\G_{0, n}))$ can be computed using \cite{F2} Lemmas
2.6 and 2.13 and Proposition 2.12. Together with the relation (cf.
\cite{F2} p. 15 for details)
$$(\sigma_0)_*\bigl(c_1(\G_{0, n})\bigr)_{\lambda}-12(\sigma_0)_*\bigl(c_1(\G_{0,
n})\bigr)_{\delta_0}+(\sigma_0)_*\bigl(c_1(\G_{0,
n})\bigr)_{\delta_1}=0,$$ this completely determine the classes
$(\sigma_0)_*\bigl(c_1(\G_{0, n})\bigr)$. Then using (\ref{auxi}) we find
$$(\sigma_0)_*(\mathfrak{a})\equiv N\Bigl(-\frac{rk(r^2k^2-3r^2k+3rk^2+2r^2+2k^2+4k-7rk-4r-10)}{(rk-r+k-3)(rk-r+k-2)}\lambda+
$$
$$+
\frac{rk(r^2k^2-3r^2k+3rk^2-8rk+2r^2+2k^2+r-k-3)}{6(rk-r+k-3)(rk-r+k-2)}\delta_0+\cdots
\Bigr),$$
$$(\sigma_0)_*(\mathfrak{b})\equiv N\Bigl(
\frac{6rk}{rk-r+k-2}\lambda-\frac{rk}{2(rk-r+k-2)}\delta_0+\cdots\Bigr),$$
and this completes the computation of the class $(\sigma_0)_*(\Xi)$
and finishes the proof.
\end{proof}

 The rather
unwieldy expressions from Theorem \ref{det} simplify nicely when
$k=2, 3$ when we obtain Theorems \ref{hyper} and \ref{trig}.

\noindent{\emph{Proof of Theorem \ref{gentype} when $g=2i+1$.}} We
construct an effective divisor on $\rr_g$ satisfying the
inequalities (\ref{inequ}) as follows: The pull-back to $\rr_{g}$ of
the Harris-Mumford divisor $\mm_{g, i+1}^1$ of curves of genus
$2i+1$ with a $\mathfrak g^1_{i+1}$ is given by the formula:
$\pi^*(\mm_{g, i+1}^1)\equiv$
$$\equiv \frac{(2i-2)!}{(i+1)! (i-1)!}\Bigl(
6(i+2)\lambda-(i+1)(\delta_0'+\delta_0^{''}+2\delta_0^{\mathrm{ram}})-\sum_{j=1}^{i}
3j(g-j)(\delta_j+\delta_{g-j}+\delta_{j: g-j})\Bigr).$$ We split
$\overline{\mathcal{D}}_{2i+1: 2}$ into boundary components of
compact type and their complement
$$\overline{\mathcal{D}}_{2i+1: 2}\equiv E+\sum_{j=1}^{i}
\bigl(a_j\delta_j+a_{g-j}\delta_{g-j}+a_{j: g-j}\delta_{j:
g-j}\bigr),$$ where $a_j, a_{g-j}, a_{j: g-j}\geq 0$  and $\Delta_j,
\Delta_{g-j}, \Delta_{j: g-j}\varsubsetneq \mbox{supp}(E)$ for
$1\leq j\leq i$, we consider the following positive linear
combination on $\rr_g$:
$$A:=\frac{i!\ (i-1)!}{(2i-1)\ (2i-3)!}\cdot
\pi^*(\mm_{2i+1, i+1}^1) +4\frac{(i!)^2}{(2i)!}\cdot E\equiv
\frac{4(3i+5)}{i+1}\
\lambda-2(\delta_0'+\delta_0^{''})-3\delta_0^{\mathrm{ram}}-\cdots,$$
where each of the coefficients of $\delta_j, \delta_{g-j}$ and
$\delta_{j: g-j}$ in the expansion of $A$ are at least
$$\frac{6(i-1)j(2i+1-j)}{(2i-1)(i+1)}\geq 2.$$ Since
$\frac{4(3i+5)}{i+1}<13$ for $i\geq 8$, the conclusion now follows
using (\ref{inequ}). For $i=7$ we find that $A\equiv
13\lambda-2(\delta_0'+\delta_0^{''})-3\delta_0^{\mathrm{ram}}-\cdots$,
hence $\kappa(\rr_{15})\geq 0$. To obtain that $\kappa(\rr_{15})\geq
1$, we use the fact that on $\mm_{15}$ there exists a Brill-Noether
divisor other than $\mm_{15, 8}^1$, namely the divisor $\mm_{15,
14}^3$ of curves $[C]\in \cM_{15}$ with a $\mathfrak g^3_{14}$. This
divisor has the same slope $s(\mm_{15, 14}^3)=s(\mm_{15,
8}^1)=27/4$, but $\mbox{supp}(\mm_{15, 14}^3)\neq
\mbox{supp}(\mm_{15, 8}^1)$. It follows that there exist constants
$\alpha, \beta, \gamma, m \in \mathbb Q_{>0}$ such that
$$\alpha\cdot E+\beta \cdot \pi^*(\mm_{15, 8}^1)\equiv \alpha\cdot
E+\gamma \cdot \pi^*(\mm_{15, 14}^3) \in |mK_{\rr_{15}}|.$$ Thus we
have found distinct multicanonical divisors on $\mm_{15}$,  that is,
$\kappa(\mm_{15})\geq 1$.  \hfill  $\Box$
\begin{remark}\label{sl15}
The same numerical argument shows that if one replaces $\mm_{15,
8}^1$ with any divisor $D\in \mbox{Eff}(\mm_{15})$ with
$s(D)<s(\mm_{15, 8}^1)=27/4$, then $\rr_{15}$ is of general type.
Any counterexample to the Slope Conjecture on $\mm_{15}$ makes
$\rr_{15}$ of general type.
\end{remark}

\section{Koszul cohomology of Prym canonical curves}

We recall that for a curve $C$, a line bundle $L\in \mbox{Pic}^d(C)$
and integers $i, j\geq 0$, the Koszul cohomology group $K_{i, j}(C,
L)$ is obtained from the complex
$$\wedge^{i+1} H^0(L)\otimes H^0(L^{\otimes
(j-1)})\stackrel{d_{i+1, j-1}}\longrightarrow \wedge^i H^0(L)\otimes
H^0(L^{\otimes j})\stackrel{d_{i, j}}\longrightarrow \wedge^{i-1}
H^0(L)\otimes H^0(L^{\otimes (j+1)}),$$ where the maps are the
Koszul differentials (cf. \cite{GL}). There is a well-known
connection between Koszul cohomology groups and Lazarsfeld bundles.
Assuming that $L$ is globally generated, a diagram chasing argument
involving exact sequences of the type
$$ 0\longrightarrow \wedge^a M_L\otimes L^{\otimes b} \rightarrow
\wedge^a H^0(L)\otimes L^{\otimes b}\longrightarrow \wedge^{a-1}
M_L\otimes L^{\otimes (b+1)}\longrightarrow 0$$ for various $a,
b\geq 0$ , yields the following identification (see also \cite{GL}
Lemma 1.10)
\begin{equation}\label{koszul}
 K_{i, j}(C, L)=\frac{H^0(C, \wedge^i
M_L\otimes L^{\otimes j})}{\mathrm{Image} \{\wedge^{i+1} H^0(C,
L)\otimes H^0(C, L^{\otimes (j-1)})\}}\ .
\end{equation}
We fix  $[C, \eta]\in \cR_g$, set $L:=K_C\otimes \eta\in
W^{g-2}_{2g-2}(C)$ and consider the \emph{Prym-canonical} map
$C\stackrel{|L|}\rightarrow \PP^{g-2}$. We denote by
$\mathcal{I}_C\subset \OO_{\PP^{g-2}}$ the ideal sheaf of the
Prym-canonical curve.

By analogy with \cite{F2} we study the Koszul stratification of
$\cR_g$ and define the strata
$$\cU_{g, i}:=\{[C, \eta]\in \cR_g: K_{i, 2}(C, K_C\otimes \eta)\neq
\emptyset\}.$$ Using (\ref{koszul}) we write the series of
equivalences
$$[C, \eta]\in \cU_{g, i}\Leftrightarrow H^1(C, \wedge^{i+1}
M_L\otimes L)\neq \emptyset\Leftrightarrow h^0(C, \wedge^{i+1}
M_L\otimes L)>$$ $$>{g-2\choose
i+1}\Bigl(-\frac{(i+1)(2g-2)}{g-2}+(g-1)\Bigr).$$ Next we write down
the exact sequence
$$0\longrightarrow H^0( \wedge^{i+1}
M_{\PP^{g-2}}(1))\stackrel{a}\longrightarrow H^0(C, \wedge^{i+1}
M_L\otimes L)\longrightarrow H^1( \wedge^{i+1} M_{\PP^{g-2}}\otimes
\mathcal{I}_C(1))\longrightarrow 0,$$ and then also
$$\mathrm{Coker}(a)=H^1\bigl(\PP^{g-2}, \wedge^{i+1} M_{\PP^{g-2}}\otimes
\mathcal{I}_C(1)\bigr)=H^0\bigl(\PP^{g-2}, \wedge^i
M_{\PP^{g-2}}\otimes \mathcal{I}_C(2)\bigr).$$ Using the well-known
fact that $h^0(\PP^{g-2}, \wedge^{i+1} M_{\PP^{g-2}}(1))={g-1\choose
i+2}$ (use for instance the  Bott vanishing theorem), we end-up with
the following equivalence:
\begin{equation}\label{determinantal} [C, \eta]\in \mathcal{U}_{g,
i}\Leftrightarrow h^0(\PP^{g-2}, \wedge^{i} M_{\PP^{g-2}}\otimes
\mathcal{I}_C(2))> {g-3\choose i}\frac{(g-1)(g-2i-6)}{i+2}.
\end{equation}

\begin{proposition}\label{ni} (1)
For $g<2i+6$, we have that $K_{i, 2}(C, K_C\otimes \eta)\neq
\emptyset $ for any $[C, \eta]\in \cR_g$, that is, the
Prym-canonical curve $C\stackrel{|K_C+\eta|}\longrightarrow
\PP^{g-2}$ does not satisfy property $(N_i)$.

\noindent (2) For $g=2i+6$, the locus $\cU_{g, i}$ is a virtual
divisor on $\cR_g$, that is, there exist vector bundles $\G_{i, 2}$
and $\H_{i, 2}$ over $\textbf{R}_g$ such that $\mathrm{rank}(\G_{i,
2})=\mathrm{rank}(\H_{i, 2})$, together with a bundle morphism
$\phi:\H_{i, 2}\rightarrow \G_{i, 2}$ such that $\cU_{g, i}$ is the
degeneracy locus of $\phi$.
\end{proposition}
\begin{proof} Part (1) is an immediate consequence of (\ref{determinantal}),
since we have the equivalence
$$K_{i, 2}(C, K_C\otimes \eta)=0\Leftrightarrow h^0(\PP^{g-2}, \wedge^{i} M_{\PP^{g-2}}\otimes
\mathcal{I}_C(2))= {g-3\choose i}\frac{(g-1)(g-2i-6)}{i+2}.$$ For
part (2) one constructs two vector bundles $\G_{i, 2}$ and $\H_{i,
2}$ over $\textbf{R}_g$ having fibres
$$\G_{i, 2}[C, \eta]=H^0(C, \wedge^i M_{K_C\otimes \eta}(2))\  \mbox{ and } \H_{i, 2}[C,
\eta]=H^0(\PP^{g-2}, \wedge^i M_{\PP^{g-2}}(2)).$$ There is a
natural morphism $\phi:\H_{i, 2}\rightarrow \G_{i, 2}$ given by
restriction. We have that
$$\mbox{rank}(\G_{i, 2})={g-2\choose
i}\Bigl(-\frac{i(2g-2)}{g-2}+3(g-1)\Bigr) \ \mbox{ and }
\mbox{rank}(\H_{i, 2})=(i+1){g\choose i+2}$$ and the condition that
$\mbox{rank}(\G_{i, 2})=\mbox{rank}(\H_{i, 2})$ is equivalent to
$g=2i+6$.
\end{proof}

We describe a set-up that will be  used to define certain
tautological sheaves over $\per_g$ and compute the class
$[\overline{\cU}_{g, i}]^{\small{virt}}$. We use the notation from
Subsection 1.1, in particular from Proposition \ref{syzi} and recall
that $f:\mathcal{X}\rightarrow \per_g$ is the universal Prym curve,
$\P\in \mbox{Pic}(\mathcal{X})$ denotes the universal Prym line
bundle and $\cN_i=f_*(\omega_f^{\otimes i}\otimes \P^{\otimes i})$.
We denote by $T:=\E_0^{''}\cap \mbox{Sing}(f)$ the codimension $2$
subvariety corresponding to Wirtinger covers $[C_{yq}, \eta\in
\mathrm{Pic}^0(C_{yq})[2], \nu(y)=\nu(q)]\in \mathcal{X}$ \ (where
$\nu^*(\eta)=\OO_C$), with the marked point being the node of the
underlying curve $C_{yq}$. Let us fix a point $[X:=C_{yq}, \eta,
\beta]\in \widetilde{\Delta}_0'\cup \widetilde{\Delta}_0^{''}$ where
as usual $\nu: C\rightarrow X$ is the normalization map. Then we
have an identification
\begin{equation}\label{twistedhodge}
\cN_1[X,
\eta, \beta]=\mbox{Ker}\bigl\{H^0\bigl(C, \omega_C(y+q)\otimes
\eta_C\bigr)\rightarrow (\nu_*\OO_C/\OO_X)\otimes \omega_X\otimes
\eta \cong \mathbb C_{y\sim q}\bigr\}, \end{equation} where the map
is given by taking the difference of residues at $y$ and $q$. Note
that when $\eta_C=\OO_C$, that is when$[X, \eta, \beta]\in
\widetilde{\Delta}_0^{''}$, we have that $\cN_1[X, \eta,
\beta]=H^0(C, \omega_C)$. For a point $$[X=C\cup_{\{y, q\}} E,\
\eta_C\in \sqrt{\OO_C(-y-q)}, \eta_E]\in
\widetilde{\Delta}_0^{\mathrm{ram}}$$ we have an identification
\begin{equation}
\label{twistedhodge2} \cN_1[X, \eta, \beta]=\mbox{Ker}\bigl\{H^0(C,
\omega_C(y+q)\otimes \eta_C)\oplus H^0(E, \OO_E(1)) \rightarrow
(\omega_X\otimes \eta)_{y, q}\cong \mathbb C^2_{y, q}\bigr\}.
\end{equation} We set
$$\cM:=\mbox{Ker}\{f^*(\cN_1)\rightarrow \omega_f\otimes \P\}.$$
>From the discussion above it is clear that the image of
$f^*(\cN_1)\rightarrow \omega_f\otimes \P$ is $\omega_f\otimes
\P\otimes \mathcal{I}_T$. Since $T\subset \mathcal{X}$ is smooth of
codimension $2$ it follows that $\cM$ is locally free. For $a, b\geq
0$, we define the sheaf $\E_{a, b}:=f_*(\wedge^a \cM\otimes
\omega_f^{\otimes b}\otimes \P^{\otimes b})$ over $\per_g$. Clearly
$\E_{a, b}$ is locally free. We have that $\E_{0, b}=\cN_b$ for
$b\geq 0$, and we always have left-exact sequences
\begin{equation}\label{gi}
0\longrightarrow \E_{a, b}\longrightarrow \wedge^a \E_{0, 1}\otimes
\E_{0, b}\longrightarrow \E_{a-1, b+1}, \end{equation} which are
right-exact off the divisor $\widetilde{\Delta}_0^{''}$ (to be
proved later). We then define inductively a sequence of vector
bundles $\{\H_{a, b}\}_{a, b\geq 0}$ over $\per_g$ in the following
way: We set $\H_{0, b}:= \mbox{Sym}^b \cN_1$ for each $b\geq 0$.
Then having defined $\H_{a-1, b}$ for all $b\geq 0$, we define the
vector bundle $\H_{a, b}$ by the exact sequence
\begin{equation}\label{hi}
0\longrightarrow \H_{a, b}\longrightarrow \wedge^a \H_{0, 1}\otimes
\mbox{Sym}^b \H_{0, 1}\longrightarrow \H_{a-1, b+1}\longrightarrow
0.
\end{equation}
For a point  $[X, \eta, \beta]\in \pr_g$, if we use the
identification $H^0(X, \omega_X\otimes \eta)=H^0(\PP^{g-2},
\OO_{\PP^{g-2}}(1))$,  we have a natural identification of the fibre
$$\H_{a, b}[X, \eta, \beta]=H^0(\PP^{g-2}, \wedge^a
M_{\PP^{g-2}}(b)).$$ By induction on $a\geq 0$, there exist vector
bundle morphisms $\phi_{a, b}:\H_{a, b}\rightarrow \E_{a, b}$.

\begin{proposition}
For $b\geq 2$ and $a\geq 0$ we have the vanishing of the higher
direct images
$$R^1 f_*\bigl(\wedge^a\cM\otimes \omega_f^{\otimes b}\otimes
\P^{\otimes b}\bigr)_{| \ \textbf{R}_g\cup \widetilde{\Delta}_0'\cup
\widetilde{\Delta}_0^{\mathrm{ram}}}=0.$$ It follows that the
sequences (\ref{gi}) are right-exact off the divisor
$\widetilde{\Delta}_0^{''}$ of $\per_g$.
\end{proposition}
\begin{proof} Over the locus $\textbf{R}_g$ the vanishing is a
consequence of Proposition \ref{lazvanish}. For simplicity we prove
that $R^1 f_*\bigl(\wedge^a\cM\otimes \omega_f^{\otimes b}\otimes
\P^{\otimes b}\bigr)\otimes
\OO_{\widetilde{\Delta}_0^{\mathrm{ram}}}=0$, the vanishing over
$\widetilde{\Delta}_0'$ being similar. We fix a point
$[X=C\cup_{\{y, q\}}E, \eta_C, \eta_E]\in
\widetilde{\Delta}_0^{\mathrm{ram}}$, with $\eta_C^{\otimes
2}=\OO_C(-y-q)$, $\eta_E=\OO_E(1)$ and set $L:=\omega_X\otimes
\eta\in \mbox{Pic}^{2g-2}(X)$. We show that $H^1(X, \wedge^a
M_{L}\otimes L^{\otimes b})=0$ for all $a\geq 0$ and $b\geq 2$. A
Mayer-Vietoris argument shows that it suffices to prove that
\begin{equation}\label{auxv1}
H^1(C, \wedge^a M_L\otimes L^{\otimes b}\otimes \OO_C)=0, \  \mbox{
} H^1(E, \wedge^a M_L\otimes L^{\otimes b}\otimes \OO_E)=0, \mbox{
and }
\end{equation}
\begin{equation}\label{auxv2}
H^1(C, \wedge^a M_L\otimes L^{\otimes b} \otimes \OO_C(-y-q))=0.
\end{equation}
For $L_C:=L\otimes \OO_C=K_C(y+q)\otimes \eta_C$ and
$L_E:=L_E\otimes \OO_E$, we write the exact sequences
$$0\longrightarrow H^0(C, L_C(-y-q))\otimes \OO_E \longrightarrow
M_L\otimes \OO_E \longrightarrow M_{L_E}\longrightarrow 0,\ \mbox{
and}$$
$$0\longrightarrow H^0(E, L_E(-y-q))\otimes \OO_C\longrightarrow
M_L\otimes \OO_C\longrightarrow M_{L_C}\longrightarrow 0,$$ and we
find that  $M_{L}\otimes \OO_C=M_{L_C}$ while obviously
$M_{L_E}=\OO_E(-1)$. We conclude that the statements (\ref{auxv1})
and (\ref{auxv2}) for all $a\geq 0$ and $b\geq 2$ can be reduced to
showing that
$$H^1(C,\wedge^a M_{L_C}\otimes L_C^{\otimes b})=H^1(C, \wedge^a
M_{L_C}\otimes L_C^{\otimes b}\otimes \OO_C(-y-q))=0, \mbox{ for all
} a\geq 0, b\geq 2.$$ This is now a direct application of
Proposition \ref{lazvanish}.
\end{proof}

\noindent \emph{Proof of Theorem \ref{prymkoszul}.} We have
constructed the vector bundle morphism $\phi_{i, 2}:\H_{i,
2}\rightarrow \E_{i, 2}$ over $\per_g$. For $g=2i+6$ we have that
$\mbox{rank}(\H_{i, 2})=\mbox{rank}(\E_{i, 2})$ and the virtual
Koszul class $[\overline{\mathcal{U}}_{g, i}]^{virt}$ is given by
$c_1(\E_{i, 2}-\H_{i, 2})$. We recall that for a rank $e$ vector
bundle $\mathcal{E}$ over a variety $X$ and for $i\geq 1$, we have
the formulas $c_1(\wedge^i \mathcal{E})={e-1\choose i-1}c_1(\E)$ and
$c_1(\mbox{Sym}^i(\E))={e+i-1\choose e} c_1(\E)$. Using  (\ref{gi})
we find that there exists a constant $\alpha\geq 0$ such that
$$c_1(\E_{i, 2})-\alpha\cdot \delta_0^{''}= \sum_{l=0}^i (-1)^l c_1(\wedge^{i-l}\E_{0,
1}\otimes \E_{0, l+2})=\sum_{l=0}^i (-1)^l{ g-1 \choose i-l}
c_1(\E_{0, l+2})+$$ $$+\sum_{l=0}^i (-1)^l
\bigl((g-1)(2l+3)\bigr){g-2 \choose i-l-1}c_1(\E_{0, 1}),$$ while a
repeated application of the exact sequence (\ref{hi}) gives that
$$c_1(\H_{i, 2})=\sum_{l=0}^i (-1)^l c_1(\wedge^{i-l} \H_{0,
1}\otimes \mbox{Sym}^{l+2} \H_{0, 1})=$$ $$=\sum_{l=0}^i(-1)^l\Bigl(   {g-1\choose
i-l}c_1(\mbox{Sym}^{l+2}(\H_{0, 1}))+{g+l\choose
l+2}c_1(\wedge^{i-l}\H_{0, 1})\Bigr)$$
$$=\sum_{l=0}^i (-1)^l\Bigl( {g-1 \choose i-l}{g+l \choose
g-1}+{g+l\choose l+2}{g-2\choose i-l-1}\Bigr)c_1(\H_{0, 1}),$$ with
 $\E_{0, 1}=\H_{0, 1}=\cN_1$ and $\E_{0, l+2}=\cN_{l+2}$ for $l\geq
0$. Proposition \ref{syzi} finishes the proof. \hfill $\Box$

Comparing these formulas to the canonical class of $\rr_g$, one obtains
that $\rr_g$ is of general type for $g>12$.

\section{Effective divisors on $\rr_g$}
We now use in an essential way results from \cite{F3} to produce
myriads of effective divisors on $\rr_g$. This construction, though
less explicit than that of $\overline{\cU}_{2i+6}$ and
$\overline{\mathcal{D}}_{g: k}$, is still very effective and we use
it to show $\rr_{18}, \rr_{20}$ and $\rr_{22}$ are of general type.

We consider the morphism $\chi:\rr_g\rightarrow \mm_{2g-1}$ given by
$\chi([C, \eta]):=[\tilde{C}]$, where $f:\tilde{C}\rightarrow C$ is
the \'etale double cover determined by $\eta$. Thus one has
$$f_*\OO_{\tilde{C}}=\OO_C\oplus \eta\  \mbox{ and } \ H^i(\tilde{C},
f^*L)=H^i(C, L)\oplus H^i(C, L\otimes \eta) \mbox{ for any } L\in
\mbox{Pic}(C), \ i=0, 1.$$ The pullback map $\chi^*$ at the level of
Picard groups has been determined by M. Bernstein in \cite{Be}
Lemma 3.1.3. We record her results:
\begin{proposition}\label{mira} The pullback map
$\chi^*:\mathrm{Pic}(\mm_{2g-1})\rightarrow \mathrm{Pic}(\rr_{g})$ is
given as follows:
$$\chi^*(\lambda)=2\lambda-\frac{1}{4}\delta_{0}^{\mathrm{ram}},\
\chi^*(\delta_0)=\delta_0^{\mathrm{ram}}+2\big(\delta_0'+\delta_0^{''}+\sum_{i=1}^{[g/2]}
\delta_{i: g-i}\bigr),\ \chi^*(\delta_i)=2\delta_{g-i} \mbox{ for }
1\leq i\leq g-1.$$
\end{proposition}
\begin{proof}
The formula for $\chi^*(\delta_i)$ when $1\leq i\leq g-1$ is
immediate. To determine $\chi^*(\lambda)$ one notices that
$\chi^*((\kappa_1)_{\mm_{2g-1}})=2(\kappa_1)_{\rr_g}$ and the rest
follows from Mumford's formulas
$(\kappa_1)_{\mm_{2g-1}}=12\lambda-\delta\in \mbox{Pic}(\mm_{2g-1})$
and $(\kappa_1)_{\rr_g}=12\lambda-\pi^*(\delta)\in
\mbox{Pic}(\rr_g)$.
\end{proof}
We set the integer $g':=1+\frac{g-1}{g}{2g\choose g-1}$. In [F3] we
have studied the rational map $$\phi: \mm_{2g-1}\dashrightarrow 
\mm_{1+\frac{g-1}{g}{2g\choose g-1}}, \ \mbox{
 }\  \phi[Y]:=W^1_{g+1}(Y),$$
 and  determined the pullback map at the level of  divisors
$\phi^*:\mbox{Pic}(\mm_{g'})\rightarrow \mbox{Pic}(\mm_{2g-1})$. In
particular, we proved that if $A\in \mbox{Pic}(\mm_{g'})$ is a
divisor of slope $s(A)=s$, then the slope of the pullback
$\phi^*(A)$ is equal to (cf. \cite{F3} Theorem 0.2)
\begin{equation}\label{gisp}
s\bigl(\phi^*(A)\bigr)=6+\frac{8g^3s-32g^3-19g^2s+66g^2+6gs-16g+3s+6}{(g-1)(g+1)(g^2s-2gs-4g^2+7g+3)}.
\end{equation}
To obtain effective divisors of small slope on $\rr_g$ we shall
consider pullbacks $(\phi\chi )^*(A)$, where $A\in
\mbox{Ample}(\mm_{g'})$. (Of course, one can consider the cone
$\chi^*(\mbox{Ample}(\mm_{2g-1}))$, but a quick look at Proposition
(\ref{mira}) shows that it is impossible to obtain in this way
divisors on $\rr_g$ satisfying the inequalities (\ref{inequ}).
Pulling back merely \emph{effective} divisors $\mm_{2g-1}$ rather
than ample ones, is problematic since $\chi(\rr_g)$ tends to be
contained in many geometric divisors on $\mm_{2g-1}$.)  In order for
the pullbacks $\chi^*\phi^*(A)$ to be well-defined as effective
divisors on $\rr_g$ we  prove the following result:
\begin{proposition}
If $\mathrm{dom}(\phi)\subset \mm_{2g-1}$ is the domain of
definition of the rational morphism $\phi:\mm_{2g-1}\rightarrow
\mm_{g'}$, then $\chi(\rr_g)\cap \mathrm{dom}(\phi)\neq \emptyset$.
It follows that for any ample divisor $A\in
\mathrm{Ample}(\mm_{g'})$, the pullback $\chi^*\phi^*(A)\in
\mathrm{Eff}(\rr_g)$ is well-defined.
\end{proposition}
\begin{proof} We take a general point $[C\cup _y E, \eta_C=\OO_C,
\eta_E]\in \Delta_{1}\subset \rr_g$. The corresponding admissible
double cover is then $f:C_1\cup _{y_1}\widetilde{E}\cup _{y_2}
C_2\rightarrow C\cup_y E$, where $[C_1, y_1]$ and $[C_2, y_2]$ are
copies of $[C, y]$ mapping isomorphically to $[C, y]$, and
$f:\widetilde{E}\rightarrow E$ is the \'etale double cover induced
by the torsion point $\eta_E\in \mbox{Pic}^0(E)[2]$. We have that
$C_i\cap \widetilde{E}=\{y_i\}$, where
$f_{\widetilde{E}}(y_1)=f_{\widetilde{E}}(y_2)=y$. Thus $\chi[C\cup
E, \OO_C, \eta_E]:=[C_1\cup _{y_1} \widetilde{E}\cup _{y_2} C_2]$,
where $y_1, y_2\in \widetilde{E}$ are such that
$\OO_{\widetilde{E}}(y_1-y_2)$ is a $2$-torsion point in
$\mbox{Pic}^0(\widetilde{E})$.

Suppose now that $X:=C_1\cup _{y_1} E\cup_{y_2} C_2$ is a curve of
compact type such that $[C_i, y_i]\in \cM_{g-1, 1}$ ($i=1, 2$) and
$[E, y_1, y_2]\in \cM_{1, 2}$ are all Brill-Noether general. In
particular, the class  $y_1-y_2\in \mbox{Pic}^0(E)$ is not  torsion.
Then $\phi([X]):=[\overline{W}^1_{g+1}(X)]$ is the variety of limit
linear series $\mathfrak g^1_{g+1}$ on $X$. The general point of
each irreducible component of $\overline{W}^1_{g+1}(X)$ corresponds
to a refined linear series $l$ on $X$ satisfying the following
compatibility conditions in terms of  Brill-Noether numbers (see
also \cite{EH}, \cite{F3}):
\begin{equation}\label{limitlin}
1=\rho(l_{C_1}, y_1)+\rho(l_{C_2}, y_2)+\rho(l_E, y_1, y_2)=1 \
\mbox{ and }\rho(l_{C_1}, y_1), \rho(l_{C_2}, y_2), \rho(l_E, y_1,
y_2)\geq 0. \end{equation} If $\rho(l_{C_2}, y_2)=1$,  we find two
types of components of $\overline{W}^1_{g+1}(X)$ which we describe:
Since $\rho(l_{C_1}, y_1)=0$, there exists an integer $0\leq a\leq
g/2$ such that $a^{l_{C_1}}(y_1)=(a, g+2-a)$. On $E$ there are two
choices for $l_E\in G^1_{g+1}(E)$ such that $a^{l_E}(y_1)=(a-1,
g+1-a)$. Either
 $a^{l_E}(y_2)=(a, g+1-a)$ (there is a unique such $l_E$), and then
 $l_{C_2}$ belongs to the connected curve $T_a:=\{l_{C_2}\in G^1_{g+1}(C_2): a^{l_{C_2}}(y_2)\geq
 (a, g+1-a)\}$, or
else, $a^{l_E}(y_2)=(a-1, g+2-a)$ (again, there is a unique such
$l_E$), and then the $C_2$-aspect of $l$ belongs to the curve
$T_a':=\{l_{C_2}\in G^1_{g+1}(C_2): a^{l_{C_2}}(y_2)\geq (a-1,
g+2-a)\}$.  Thus $\{l_{C_1}\}\times T_a$ and $\{l_{C_2}\}\times
 T_a'$ are irreducible components of $\overline{W}_{g+1}^1(X)$.
When $\rho(l_E, y_1, y_2)=1$, then there are three types of
irreducible components of $\overline{W}_{g+1}^1(X)$ corresponding to
the cases
$$a^{l_E}(y_1)=(a-1, g+1-a), \ a^{l_E}(y_2)=(a-1, g+1-a)\  \mbox{ for } 0\leq a\leq g/2,$$
$$a^{l_E}(y_1)=(a-1, g+1-a), \ a^{l_E}(y_2)=(a, g-a)  \mbox{ for } 1\leq a\leq (g-1)/2, \ \mbox{  and
}$$
$$a^{l_E}(y_1)=(a-1, g+1-a), \ a^{l_E}(y_2)=(a-2, g+2-a) \ \mbox{ for } 2\leq a \leq (g-1)/2.$$
Finally, the case $\rho(l_{C_1}, y_1)=1$ is identical to the case
$\rho(l_{C_2}, y_2)=1$ by reversing the role of the curves $C_1$ and
$C_2$. The singular points of $\overline{W}_{g+1}^1(X)$ correspond
to (necessarily) crude limit $\mathfrak g^1_{g+1}$'s satisfying
$\rho(l_{C_1}, y_1)=\rho(l_{C_2}, y_2)=\rho(l_{E}, y_1, y_2)=0$. For
such $l$, there must exist two irreducible components of $X$, say
$Y$ and $Z$, for which $Y\cap Z=\{x\}$ and such that
$a_0^{l_Y}(x)+a_1^{l_Z}(x)=g+2$ and $a_1^{l_Y}(x)+a_0^{l_Z}(x)=g+1$.
The point $l$ lies precisely on the two irreducible components of
$\overline{W}_{g+1}^1(X)$: The one corresponding to refined limit
$\mathfrak g^1_{g+1}$ with vanishing sequence on $Y$ equal to
$(a_0^{l_Y}(x)-1, a_1^{l_Y}(x))$, and the one with vanishing
$(a_0^{l_Z}(x), a_1^{l_Z}(x)-1)$ on $Z$. Thus
$\overline{W}_{g+1}^1(X)$ is a stable curve of compact type, so
$[X]\in \mbox{dom}(\phi)$. Using \cite{F3}, this set-theoretic
description applies to the image under $\phi$ of any point
$[C_1\cup_{y_1} E\cup _{y_2} C_2]$, in particular to $[C_1\cup_{y_1}
\widetilde{E}\cup _{y_2} C_2]=\chi([C\cup _y E]).$ We have showed
that $\chi(\Delta_1)\cap \mbox{dom}(\phi)\neq \emptyset$.
\end{proof}

\noindent {\emph{Proof of Theorem \ref{gentype} for genus $g=18, 20,
22$.}} We construct an effective divisor on $\rr_g$ which satisfies
the inequalities (\ref{inequ}) and which is of the form
$$\mu\pi^*(D)+\epsilon\chi^* \phi^*(A)=\alpha
\lambda-2(\delta_0'+\delta_0^{''})-3\delta_{0}^{\mathrm{ram}}-\sum_{i=1}^{[g/2]}(b_i\delta_i+b_{g-i}\delta_{g-i}
+b_{i: g-i}\delta_{i: g-i}),$$ where $A\equiv s\lambda-\delta\in
\mbox{Pic}(\mm_{g'})$ is an ample class (which happens precisely
when $s>11$, cf. \cite{CH}), $D\in \mbox{Eff}(\mm_g)$ and $\mu,
\epsilon>0$ and $\alpha<13$. We solve this linear system using
Proposition \ref{mira} and find that we must have
$$\epsilon=\frac{8}{12-s(\phi^*(A))} \ \mbox{  and }
\ \mu=\frac{16-2s(\phi^*(A))}{12-s(\phi^*(A))}.$$ To
conclude that $\rr_g$ is of general type, it suffices to check that
the inequality
$$\alpha=\frac{8s\bigl(\phi^*(A)\bigr)}{12-s\bigl(\phi^*(A)\bigr)}+\Bigl(6+\frac{12}{g+1}\Bigr)\frac{16-2s\bigl(\phi^*(A)\bigr)}
{12-s\bigl(\phi^*(A)\bigr)}<13$$ has a solution $s=s(A)\geq 11$.
Using (\ref{gisp}), we find that this is the case for $g\geq 18$.
\hfill $\Box$

\section{The enumerative geometry of $\rr_g$ in small genus}

In this Section we describe the divisors $\mathcal{D}_{g: k}$ and
$\cU_{g, i}$ for small $g$. We start with the case $g=3$. This
result has been first obtained by M. Bernstein \cite{Be} Theorem
3.2.3 using test curves inside $\rr_3$. Our method is more direct
and uses the identification of cycles $C-C=\Theta_{Q_C}\subset
\mbox{Pic}^0(C)$, valid for all curves $[C]\in \cM_3$.

\begin{theorem}\label{genu3}
The divisor $\mathcal{D}_{3: 2}=\{[C, \eta]\in \cR_3: \eta\in C-C\}$
is equal to the locus of \'etale double covers
$[\tilde{C}\stackrel{f}\rightarrow C]\in \cR_3$ such that
$[\tilde{C}]\in \cM_5$ is hyperelliptic. We have the equality of
cycles $\overline{\mathcal{D}}_{3: 2}\equiv
8\lambda-\delta_0'-2\delta_0^{''}-\frac{3}{2}\delta_0^{\mathrm{ram}}-6\delta_1-4\delta_2-2\delta_{1:
2}\in \mathrm{Pic}(\rr_3)$. Moreover,
$$\pi_*(\overline{\mathcal{D}}_{3: 2})\equiv 56\cdot \mm_{3, 2}^1
=56\cdot(9\lambda-\delta_0-3\delta_1)\in \mathrm{Pic}(\mm_3).$$ This
equality corresponds to the fact that for an \'etale double cover
$f:\tilde{C}\rightarrow C$, the source $\tilde{C}$ is hyperelliptic
if and only if $C$ is hyperelliptic and $\eta\in C-C\subset
\mathrm{Pic}^0(C)$.
\end{theorem}
\begin{proof} We use the set-up from Theorem \ref{det} and recall that there
exists a vector bundle morphism $\phi: \H\otimes \cA_{0,
0}\rightarrow \cA_{0, 1}$ over $\rer_3^0$ such that $Z_1(\phi)\cap
\cR_3=\mathcal{D}_{3: 2}$. Here $\H=\pi^*(\mathbb E)$, $\cA_{0,
0}[X, \eta, \beta]=H^0(X, \omega_X\otimes \beta)$ and $\cA_{0, 1}[X,
\eta, \beta]=H^0(X, \omega_X^{\otimes 2}\otimes \beta)$, for each
point $[X, \eta, \beta]\in \pr_g$. Using (\ref{twistedhodge}) and
(\ref{twistedhodge2}) we check that both $\phi_{| \Delta_0'}$ and
$\phi_{| \Delta_0^{\mathrm{ram}}}$ are generically non-degenerate.
Over a point $t=[C_{yq}, \ \eta, \beta]\in \Delta_0^{''}$
corresponding to a Wirtinger covering (i.e. $\nu: C\rightarrow
C_{yq}$, with $[C]\in \cM_2$ and $\nu^*(\eta)=\OO_C$), we have that
$$\phi(t): H^0(C, K_C)\otimes H^0\bigl(C, K_C\otimes \OO_C(y+q)\bigr)\rightarrow
\cA_{0, 1}(t)\subset H^0\bigl(C, \omega_C^{\otimes 2}\otimes
\OO_C(2y+2q)\bigr).$$ From the base point free pencil trick we find
that $\mbox{Ker}(\phi(t))=H^0(C, \OO_C(y+q))$, that is, $\phi_{|
\Delta_0^{''}}$ is everywhere degenerate and the class $c_1(\cA_{0, 1}-\H
\otimes \cA_{0, 0})-\delta_{0}^{''}\in \mbox{Pic}(\rer_3^0)$ is
effective. From the formulas $\pi_*(\lambda)=63\lambda,\
\pi_*(\delta_0')=30\delta_0, \ \pi_*(\delta_0^{''})=\delta_0$ and
$\pi_*(\delta_0^{\mathrm{ram}})=16\delta_0$, we obtain that
$$s\bigl(\pi_*(c_1(\cA_{0, 1}-\H\otimes \cA_{0,
0})-\delta_{0}^{''})\bigr)=9.$$ The hyperelliptic locus $\mm_{3,
2}^1$ is the only divisor on $D\in \mbox{Eff}(\mm_3)$ with
$\Delta_i\varsubsetneq \mbox{supp}(D)$ for $i=0, 1$ and $s(D)\leq
9$, which leads to the formula $\pi_*(\overline{\mathcal{D}}_{3:
2})=56\cdot \mm_{3, 2}^1$.
\end{proof}
\begin{theorem}\label{genu5}
The divisor $\overline{\mathcal{D}}_{5:2}:=\{[C, \eta]\in \cR_5:
\eta\in C_2-C_2\}$ equals the locus of \'etale double covers
$[\tilde{C}\stackrel{f}\rightarrow C]\in \cR_5$ such that the genus
$9$ curve $\tilde{C}$ is tetragonal. We have the formula
$\overline{\mathcal{D}}_{5:
2}=14\lambda-2(\delta_0'+\delta_0^{''})-\frac{5}{2}\delta_0^{\mathrm{ram}}
-10\delta_4-4\delta_{1: 4}-\cdots \in \mathrm{Pic}(\rr_5)$.
\end{theorem}
\begin{proof} We start with an \'etale cover
$f:\tilde{C}\stackrel{2:1} \rightarrow C$ corresponding to the
torsion point $\eta=\OO_C(D-E)$, with $D, E\in C_2$. Then
$$H^0(\tilde{C}, \OO_{\tilde{C}}(f^*D))=H^0(C, \OO_C(D))\oplus H^0(C,
\OO_C(E)),$$ that is, $|f^*D|\in G^1_4(\tilde{C})$ and
$[\tilde{C}]\in \mm_{9, 4}^1$. Conversely, if $l\in
G^1_4(\tilde{C})$, then $l$ must be invariant under the involution
of $\tilde{C}$ and then $f_*(l)\in G^1_4(C)$ contains two divisors
of the type $2x+2y\equiv 2p+2q$. Then we take $\eta=\OO_C(x+y-p-q)$,
that is, $[C, \eta]\in \mathcal{D}_{5: 2}$.
\end{proof}
\begin{remark} Since $\mbox{codim}(\mm_{9, 4}^1, \mm_9)=3$ while
$\mathcal{D}_{5: 2}$ is a divisor in $\cR_3$, there seems to be a
dimensional discrepancy in Theorem \ref{genu5}. This is explained by
noting that for an \'etale double covering $f:\tilde{C}\rightarrow
C$ over a general curve $[C]\in \cM_5$, the codimension $1$
condition $\mbox{gon}(\tilde{C})\leq 5$ is equivalent to the
seemingly stronger condition $\mbox{gon}(\tilde{C})\leq 4$. Indeed,
if $l\in G^1_5(\tilde{C})$ is base point free, then $l$ is not
invariant under the involution of $\tilde{C}$ and $\mbox{dim }
|f_*l|\geq 2$ so $G^2_5(C)\neq \emptyset$, a contradiction with the
genericity assumption on $C$.
\end{remark}
\begin{theorem}\label{genu4}
The divisor $\mathcal{D}_{4: 3}=\{[C, \eta]\in \cR_4: \exists A\in
W^1_3(C) \mbox{ with } H^0(C, A\otimes \eta)\neq 0\}$ can be
identified with the locus of Prym curves $[C, \eta]\in \cR_4$ such
that the Prym-canonical model $C\stackrel{|K_C\otimes
\eta|}\longrightarrow \PP^2$ is a plane sextic curve with a triple
point. We also have the class formula
$$\overline{\mathcal{D}}_{4: 3}\equiv
8\lambda-\delta_0'-2\delta_0^{''}-\frac{7}{4}\delta_0^{\mathrm{ram}}-4\delta_3-7\delta_1-3\delta_{1:
3}-\cdots \in \mathrm{Pic}(\rr_4),$$ hence
$\pi_*(\overline{\mathcal{D}}_{4: 3})=60\cdot
\overline{\mathcal{GP}}_{4,
3}^1=60(34\lambda-4\delta_0-14\delta_1-18\delta_2)\in \mathrm{Pic}(\mm_4)$, where
$$\mathcal{GP}_{4, 3}^1\subset \cM_4:=\{C]\in \cM_4: \exists A\in W^1_3(C), A^{\otimes 2}=K_C\}$$ is the Gieseker-Petri divisor
of curves $[C]\in \cM_4$ with a vanishing theta-null.
\end{theorem}
\begin{proof} We start with a Prym curve $[C, \eta]\in \cR_4$ such
that there exists a pencil $A\in W^1_3(C)$ with $H^0(C, A\otimes \eta)\neq
0$. We claim that $A^{\otimes 2}=K_C$, that is, $[C]\in
\mathcal{GP}_{4, 3}^1$. 

Indeed, assuming the opposite, we find
\emph{disjoint} divisors $D_1, D_2\in C_3$ such that $D_1\in
|A\otimes \eta|$ and $D_2\in |K_C\otimes A^{\vee}\otimes \eta|$. In
particular, the subspaces $$H^0(C, K_C\otimes \eta(-D_i))\subset
H^0(C, K_C\otimes \eta)$$ are both of dimension $2$ and  intersect
non-trivially, that is $H^0\bigl(C, K_C\otimes
\eta(-D_1-D_2)\bigr)\neq 0$. Since $D_1+D_2\equiv K_C$, this implies
$\eta=0$, a contradiction.

The proof that the vector bundle morphism $\phi: \H \otimes \cA_{0,
0}\rightarrow \cA_{0, 1}$ constructed in the proof of Theorem
\ref{det} is degenerate with order $1$ along the divisor
$\Delta_0^{''}\subset \rr_4$ follows from (\ref{twistedhodge}). Thus
$c_1\bigl(\cA_{0, 1}-\H\otimes \cA_{0, 0}\bigr)-\delta_0^{''}\in
\mathrm{Pic}(\rr_4)$ is an effective class and its push-forward to
$\mm_4$ has slope $17/2$. The only divisor $D\in \mbox{Eff}(\mm_4)$
with $\Delta_i\varsubsetneq \mbox{supp}(D)$ for $i=0, 1, 2$ and
$s(D)\leq 17/2$, is the theta-null divisor
$\overline{\mathcal{GP}}_{4, 3}^1$ (cf. \cite{F3} Theorem 5.1).
\end{proof}
\begin{remark}
For a general point $[C, \eta]\in \cR_4$, the Prym-canonical curve
$\iota:C\stackrel{|K_C\otimes \eta|}\longrightarrow\PP^2$ is a plane
sextic with $6$ nodes which correspond to the preimages in
$\phi^{-1}(\eta)$ under the second difference map $$C_2\times C_2\rightarrow
\mbox{Pic}^0(C), \mbox{   }\ (D_1, D_2)\mapsto \OO_C(D_1-D_2).$$ Note that $W_2(C)\cdot (W_2(C)+\eta)=6$. For a
general $[C, \eta]\in \mathcal{D}_{4:3}$, the model $\iota(C)\subset
\PP^2$ has a triple point. For a hyperelliptic curve $[C]\in \cM_{4,
2}^1$, out of the $255=2^{2g}-1$ \'etale double covers of $C$,
there exist $210$ for which $C\stackrel{|K_C\otimes
\eta|}\longrightarrow \PP^2$ has an ordinary $4$-fold point and no
other singularity. The remaining $45={2g+2\choose 2}$ coverings
correspond to the case $\eta=\OO_C(x-y)$, with $x, y\in C$ being
Weierstrass points, when $|K_C\otimes \eta|$ has $2$ base points and
$\iota$ is a degree $2$ map onto a conic.
\end{remark}

\section{The singularities of the moduli space of Prym curves}

The moduli space $\rr_g$ is a normal variety with finite quotient
singularities. To determine its Kodaira dimension we consider a
smooth model $\widehat{\mathcal{R}}_g$ of $\rr_g$ and then analyze
the growth of the dimension of the spaces $H^0\bigl(\widehat
{\mathcal{R}}_g, K_{\widehat{\mathcal{R}}_g}^{\otimes l} \bigr)$ of
pluricanonical forms for all $l\geq 0$. In this section we show
that in doing so one only needs to consider forms defined on $\rr_g$
itself.
\begin{theorem}\label{thm:canforms}
We fix $g\geq 4$ and let $\widehat{\mathcal{R}}_g\rightarrow\rr_g$
be any desingularisation. Then every pluricanonical form defined on
the smooth locus $\rr_g^{\mathrm{reg}}$ of $\rr_g$ extends
holomorphically to $\widehat{\mathcal{R}}_g$, that is,  for all
integers $l\geq 0$ we have isomorphisms
\[
H^0\bigl(\rr_g^{\mathrm{reg}},K_{\rr_g}^{\otimes l}\bigr) \cong
H^0\bigl(\widehat{\mathcal{R}}_g,
K_{\widehat{\mathcal{R}}_g}^{\otimes l} \bigr).
\]
\end{theorem}
A similar statement has been proved for the moduli space $\mm_g$ of
curves cf. \cite{HM} Theorem 1, and for the moduli space
$\overline{\mathcal{S}}_g$ of spin curves, cf. \cite{lu2007} Theorem
4.1. We start by  explicitly describing the locus of non-canonical
singularities in $\rr_g$, which has codimension $2$. At a
non-canonical singularity there exist \emph{local} pluricanonical
forms that do acquire poles on a desingularisation. We show that
this situation does not occur for forms defined on the smooth locus
$\rr_g^{\text{reg}}$, and they extend holomorphically to
$\widehat{\cR}_g$.
\begin{definition}
An \emph{automorphism} of a Prym curve $\prym$  is an automorphism
$\iso\in\aut(\prymi)$ such that there exists an isomorphism of
sheaves $\gamma:\iso^*\prymii\rightarrow\prymii$ making the
following diagram commutative.
\begin{eqnarray*}
\xymatrix @!0 @R=1.3cm @C=2cm {
        (\sigma^*\eta)^{\otimes  2}  \ar[r]^-{  \gamma^{\otimes 2}}   \ar[d]_{\sigma^* \beta} &
        \eta^{\otimes 2} \ar[d]^{\beta}\\
        \sigma^* \mathcal{O}_X  \ar[r]^{\simeq} & \mathcal{O}_X
        }
\end{eqnarray*}
If $C:=st(X)$ denotes the stable model of $X$ then there is a group
homomorphism $\mbox{Aut}\prym\rightarrow\mbox{Aut}(C)$ given by
$\iso\mapsto \iso_{\stab}$. The kernel $\mbox{Aut}_0\prym$ of this
homomorphism is called the subgroup of \emph{inessential
automorphisms} of $(X, \eta, \beta)$.
\end{definition}
\begin{remark}\label{rem:auto}
The subgroup  $\auto\prym$  is isomorphic to $\{\pm
1\}^{CC(\nonex)}/\pm 1$, where $CC(\nonex)$ is the set of connected
components of the non-exceptional subcurve $\nonex$ (compare
\cite{CCC} Lemma 2.3.2 and~\cite{lu2007} Proposition~2.7). Given
$\gamma_j\in\{\pm 1\}$ for every connected component $\nonexcomp$ of
$\nonex$ consider the automorphism $\widetilde\gamma$ of
$\widetilde\prymii=\prymii_{|\nonex}$ which is multiplication by
$\gamma_j$ in every fibre over $\nonexcomp$. Then there exists a
unique inessential automorphism $\iso$ such that $\widetilde\gamma$
extends to an isomorphism $\gamma:\iso^*\prymii\rightarrow\prymii$
compatible with the morphisms $\iso^*\prymiii$ and $\prymiii$.
Considering $(-\gamma_j)_j$ instead of $(\gamma_j)_j$ gives the same
automorphism $\iso$.
\end{remark}
\begin{definition}
For a quasi-stable curve $\prymi$, an irreducible component $\comp$
is called an \emph{elliptic tail} if $p_a(C_j)=1$ and $\comp\cap
\overline{(X-\comp)}=\{p\}$. The node $p$ is then an \emph{elliptic
tail node}. A non-trivial automorphism $\iso$ of $\prymi$ is called
an \emph{elliptic tail automorphism} (with respect to $\comp$) if
$\iso_{|\prymi-\comp}$ is the identity.
\end{definition}
\begin{theorem}\label{thm:smoothlocus}
Let $\prym$ be a Prym curve of genus $g\geq 4$. The point $[X, \eta,
\beta]\in \rr_g$ is smooth if and only if $\aut\prym$ is generated
by elliptic tail involutions.
\end{theorem}

Throughout this Appendix, $X$ denotes a quasi-stable curve of genus
$g\geq 2$ and $\stab:=st(X)$ is its stable model. We denote by
$N\subset\mbox{Sing}(\stab)$ the set of exceptional nodes and
$\Delta:=\mbox{Sing}(\stab)-N$. Then $\prymi$ is the support of a
Prym curve if and only if $N$ considered as a subgraph of the dual
graph $\Gamma(\stab)$ is \emph{eulerian}, that is, every vertex of
$\Gamma(\stab)$ is incident to an even number of edges in $N$ (cf.
\cite{BCF} Proposition 0.4).

Locally at a point $[X, \eta, \beta]$, the moduli space $\rr_g$ is
isomorphic to the quotient of the versal deformation space
$\lud{\tau}$ of $\prym$ modulo the action of the automorphism group
$\aut\prym$. If $\lud{t}=\mbox{Ext}^1(\Omega_C^1, \OO_C)$ denotes
the versal deformation space of $\stab$, then the map
$\lud{\tau}\rightarrow\lud{t}$ is given by $t_i=\tau_i^2$ if
$(t_i=0) \subset \lud{t}$ is the locus where the exceptional node
$p_i\in N$ persists and $t_i=\tau_i$ otherwise. The
 morphism $\pi:\rr_g\rightarrow\mm_g$ is given locally by the map
$\lud{\tau}/\aut\prym\rightarrow\lud{t}/\mbox{Aut}(C)$.  One has the
following decomposition of the versal deformation space of $(X,
\eta, \beta)$
\[
\lud{\tau}=\bigoplus_{p_i\in
N}\CC_{\tau_i}\oplus\bigoplus_{p_i\in\Delta}\CC_{\tau_i}\oplus
\bigoplus_{\comp\subset C}H^1\bigl(\compnu,T_{\compnu}(-D_j)\bigr),
\]
where for a node $p_i\in N$ we denote by $(\tau_i=0)\subset \mathbb
C_{\tau}^{3g-3}$ the locus where the corresponding exceptional
component $E_i$ persists, while for a node $p_i\in \Delta$ we denote
by $(\tau_i=0)\subset \mathbb C_{\tau}^{3g-3}$ the locus of  those
deformations in which $p_i$ persists. Finally, for a component
$C_j\subset C$ with normalization $C_j^{\nu}$, if $D_j$ consists of
the inverse images of the nodes of $C$ under the normalization map
$C_j^{\nu}\rightarrow C_j$, the group
$H^1(\compnu,T_{\compnu}(-D_j))$ parameterizes deformations of the
pair $(\compnu,D_j)$. This decomposition is compatible with the
decomposition
$$\lud{t}=\bigl(\bigoplus_{p_i\in\mathrm{Sing}(\stab)}\CC_{t_i}\bigr)\oplus
\bigl(\bigoplus_{\comp}H^1(\compnu,T_{\compnu}(-D_j))\bigr)$$ as well
as with the actions of the automorphism groups on $\lud{\tau}$ and
$\lud{t}$, see also \cite{lu2007} pg. 5. The point $[X, \eta,
\beta]\in\rr_g$ is smooth if and only if the action of $\aut\prym$
on $\lud{\tau}$ is generated by quasi-reflections, that is, elements
$\iso\in\aut\prym$  having $1$ as an eigenvalue of multiplicity
precisely $3g-4$. Theorem \ref{thm:smoothlocus} follows from the
following proposition.
\begin{proposition}\label{prop:qr}
Let $\iso\in\aut\prym$ be an automorphism of a Prym curve $\prym$ of
genus $g\geq 4$. Then $\iso$ acts on $\lud{\tau}$ as a
quasi-reflection if and only if $\prymi$ has an elliptic tail
$\comp$ such that $\iso$ is the elliptic tail involution with
respect to $\comp$.
\end{proposition}
\begin{proof}
Let $\iso$ be an elliptic tail involution with respect to $\comp$.
The induced automorphism $\iso_\stab$ is an elliptic tail involution
of $C$ and acts on the versal deformation space $\lud{t}$ of $\stab$
as $t_1\mapsto-t_1$ and $t_i\mapsto t_i$, $i\neq 1$. Here $t_1$ is
the coordinate corresponding to the node $p_1\in \comp\cap
\overline{(C-C_j)}$. The node $p_1$ being non-exceptional, we have
that $t_1=\tau_1$ hence $\iso\cdot \tau_1=-\tau_1$. If $\tau_i=t_i
(i\neq 1)$, then $\sigma\cdot \tau_i=\tau_i$. For coordinates
$t_i=\tau_i^2$,  $\iso$ is the identity in a neighbourhood of the
corresponding exceptional component $E_i$, thus $\sigma \cdot
\tau_i=\tau_i$.

Now let $\iso\in\aut\prym$ act as a quasi-reflection with
eigenvalues $\zeta$ and $1$. As in the proof of \cite{lu2007}
Proposition 2.15, there exists a node $p_1\in \stab$ such that the
action of $\iso$ is given by  $\sigma\cdot \tau_1=\zeta\tau_1$ and
$\sigma \cdot \tau_j=\tau_j$ for $j\neq 1$. When $p_1\in N$, the
induced automorphism $\iso_\stab$ acts via $t_1\mapsto\zeta^2 t_1$
and $\sigma_C \cdot t_j=t_j$ for $j\neq 1$. If $\zeta^2\neq 1$, then
$\iso_\stab$ acts as a quasi-reflection and $p_1$ is an elliptic
tail node, which contradicts the assumption $p_1\in N$. Therefore
$\iso_\stab=\mbox{Id}_{\stab}$ and the exceptional component $E_1$
over $p_1$ is the only  component on which $\iso$ acts
non-trivially. The graph $N\subset\Gamma(\stab)$ is eulerian
 and there exists a circuit of edges in $N$ containing $p_1$.
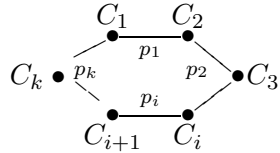
\begin{eqnarray*}
\xymatrix@!0 @R=0.8cm @C=1.0cm {
& C_1   & C_2 & \\
             C_k  \ \bullet   \ar@{-} '[ur]+D*{\bullet}_{p_k}  '[urr]+D*{\bullet}_{p_1}      '[rrr]+L*{\bullet}_{p_2}
              '[drr]+U*{\bullet}  '[dr]+U*{\bullet}_{p_i}  \ar@{--}@<1.5ex>[dr] & & & C_3   \\
             & C_{i+1}             & C_i  &
                    }
\end{eqnarray*}

By Remark~\ref{rem:auto},  $\iso$ corresponds to an element
$\pm(\gamma_j)_j\in\{\pm 1\}^{CC(\nonex)}/\pm1$. Since $\iso$ acts
non-trivially on $E_1$ we find that $\gamma_1=-\gamma_2$. In
particular, there exists $i\neq 1$ such that  $\iso$ acts
non-trivially on $E_i$. This is a contradiction which shows that the
node $p_1$ is non-exceptional, $\tau_1=t_1$ and  $\sigma_C\cdot
t_1=\zeta t_1$ and $\sigma_C \cdot t_i=t_i$ for $i\neq 1$. Thus
$\iso_\stab$ is an elliptic tail involution of $\stab$ with respect
to an elliptic tail through the node $p_1$ and $\zeta=-1$. Since
$\sigma$ fixes every coordinate corresponding to an exceptional
component of $X$, it follows that $\iso$ is an elliptic tail
involution of $\prymi$.
\end{proof}
\begin{theorem}\label{thm:cansing}
We fix $g\geq 4$. A point $[X, \eta, \beta]\in\rr_g$ is a
non-canonical singularity if and only if $\prymi$ has an elliptic
tail $\comp$ with $j$-invariant $0$ and $\prymii$ is trivial on
$\comp$.
\end{theorem}
The proof is similar to that of the analogous statement for
$\overline{\mathcal{S}}_g$ and we refer to \cite{lu2007} Theorem 3.1
for a detailed outline of the proof and background on quotient
singularities. Locally at $[X, \eta, \beta]$, the space $\rr_g$ is
isomorphic to a neighbourhood of the origin in
$\lud{\tau}/\aut\prym$. We consider the normal subgroup $H$ of
$\aut\prym$ generated by automorphisms acting as quasi-reflections
on $\lud{\tau}$. The map
$\lud{\tau}\rightarrow\lud{\tau}/H=\lud{\upsilon}$ is given by
$\upsilon_i=\tau_i^2$ if $p_i$ is an elliptic tail node and
$\upsilon_i=\tau_i$ otherwise. The automorphism group $\aut\prym$
acts on $\lud{\upsilon}$ and the quotient $\lud{\upsilon}/\aut\prym$
is isomorphic to $\lud{\tau}/\aut\prym$. Since $\aut\prym$ acts on
$\lud{\upsilon}$ without quasi-reflections the \rsbt{} criterion
applies to this new action.

We fix an automorphism  $\iso\in \mathrm{Aut}(X, \eta, \beta)$ of
order $n$ and a primitive $n$-th root of unity $\zeta_n$. If the
action of $\sigma$ on  $\lud{\upsilon}$ has eigenvalues
$\zeta_n^{a_1},\dotsc,\zeta_n^{a_{3g-3}}$ with $0\leq a_i<n$ for
$i=1, \ldots, 3g-3$, then following \cite{re2002} we define the
\emph{age} of $\iso$ by
$$\mathrm{age}(\sigma, \zeta_n):=\frac 1n\sum_{i=1}^n a_i.$$
We say that $\sigma$  satisfies the \rsbt{} inequality if
$\mathrm{age}(\sigma, \zeta_n)\geq 1$. The \rsbt{} criterion states
that $\lud{\upsilon}/\aut\prym$ has canonical singularities if and
only if  every $\iso\in\aut\prym$ satisfies the \rsbt{} inequality
(cf. \cite{re1980}, \cite{ta1982},\cite{re2002}).

\begin{proof}[Proof of the if-part of Theorem~\ref{thm:cansing}]
Let  $\prym$ be a Prym curve,  $\stab=st(X)$ and $\comp\subset X$ an
elliptic tail with $\mbox{Aut}(C_j)=\mathbb Z_6$ and we assume
$\eta_{C_j}=\OO_{C_j}$. We fix an elliptic tail automorphism
$\sigma_C$ with respect to $\comp\subset\stab$ such that
$\mbox{ord}(\iso_\stab)=6$.
 Then $\iso_\stab$ acts on $\lud{t}$ by $t_1\mapsto\zeta_6 t_1$, $t_2\mapsto\zeta_6^2 t_2$ for
an appropriate sixth root of unity $\zeta_6$ and $\sigma\cdot
t_i=t_i$ for $i\neq 1,2$. Here $t_1, t_2\in \mbox{Ext}^1(\Omega_C^1,
\OO_C)$ correspond to smoothing the node $p_1\in
C_j\cap\overline{(C-C_j)}$ and  deforming the curve $[\comp,p_1]\in
\mm_{1, 1}$ respectively. Since $\prymii_{\comp}=\OO_{\comp}$, the
automorphism $\iso_\stab$ lifts
  to an automorphism  $\sigma \in \mbox{Aut}\prym$ such that $\sigma_{\overline{X-C_j}}$ is the identity.
   Then $\iso$ acts on $\lud{\tau}$
  as $\sigma \cdot \tau_1=\zeta_6\tau_1$, $\sigma\cdot \tau_2=\zeta_6^2\tau_2$ and
  $\sigma \cdot \tau_i=\tau_i$ for $i\neq 1,2$. Since $\upsilon_1=\tau_1^2$ and $\upsilon_2=\tau_2$, the
  action of $\iso$ on
  $\lud{\upsilon}$ is $\upsilon_1\mapsto\zeta_6^2\upsilon_1$, $\upsilon_2\mapsto\zeta_6^2\upsilon_2$
  and $\upsilon_i\mapsto\upsilon_i$, $i\neq 1,2$. We compute $\mathrm{age}(\sigma, \zeta_6^2)=
 \frac 13+\frac 13+0+\dotsb+0=\frac 23<1$, that is, $[X, \eta, \beta]\in\rr_g$ is a non-canonical singularity.
Similarly, an elliptic tail automorphism of order $3$ with respect
to $\comp$ acts via $\tau_1\mapsto\zeta_3^2\tau_1$,
$\tau_2\mapsto\zeta_3\tau_2$ and $\tau_i\mapsto\tau_i$, $i\neq 1,2$,
and then for the action on $\mathbb C^{3g-3}_{\upsilon}$ as
$\upsilon_1\mapsto\zeta_3\upsilon_1$,
$\upsilon_2\mapsto\zeta_3\upsilon_2$
  and $\upsilon_i\mapsto\upsilon_i$ for $i\neq 1,2$. This gives a value of $\frac 23$ for the age.
\end{proof}

Suppose that $[X, \eta, \beta]\in\rr_g$ is a non-canonical
singularity. Then there exists an automorphism $\iso\in\aut\prym$ of
order $n$ which acts on $\lud{\upsilon}$ such that
$\mathrm{age}(\sigma, \zeta_n)<1$. Let
$p_{i_0},p_{i_1}=\iso_\stab(p_{i_0}),\dotsc,p_{i_{m-1}}=\iso_\stab^{m-1}(p_{i_0})$
be distinct nodes of $\stab$ which are cyclicly permuted by the
induced automorphism $\iso_\stab$  and $p_{i_j}$ is \emph{not} an
elliptic tail node. The action of $\iso$ on the subspace
$\bigoplus_j\CC_{\tau_{i_j}}\subset \mathbb C_{\tau}^{3g-3}$ is
given by the matrix
\[
B=\begin{pmatrix}
0&c_1&&\\
\vdots&&\ddots&\\
0&&&c_{m-1}\\
c_{m}&0&\dotsb&0
\end{pmatrix}
\]
for appropriate scalars $c_j\neq 0$. The pair
$\bigl(\prym,\iso\bigr)$ is said to be \emph{singularity reduced} if
for every such cycle we have that $\prod_{j=1}^m c_j\neq 1$.
\begin{proposition}(\cite{HM}, \cite{lu2007} $\rm{Proposition \ 3.6}$)
\label{prop:pair}
There exists a deformation $(\prymi',\prymii',\prymiii')$ of $\prym$
such that $\iso$ deforms to an automorphism
$\iso'\in\aut(\prymi',\prymii',\prymiii')$ and the nodes of every
cycle of nodes as above with $\prod_{j=1}^m c_j=1$ are smoothed. The
pair $\bigl((\prymi',\prymii',\prymiii'),\iso'\bigr)$ is then
singularity reduced and the action of $\iso$ on $\lud{\upsilon}$ and
that of $\iso'$ on $\lud{\upsilon'}$ have the same eigenvalues and
hence the same age.
\end{proposition}
We fix a singularity reduced pair $\bigl(\prym,\iso\bigr)$ with
$n:=\mbox{ord}(\sigma)\geq 2$ and assume that $\mathrm{age}(\iso,
\zeta_n)<1$. We denote this
 assumption by $(\star)$. Using \cite{lu2007} Proposition 3.7 we obtain that if $(\star)$ holds,
 the induced automorphism $\iso_\stab$ of $C=st(X)$ fixes every node with the possible exception of two nodes which are
interchanged.
\begin{proposition}
If $(\star)$ holds, then $\iso_\stab$ fixes all components of the
stable model $C$ of $X$.
\end{proposition}
\begin{proof}
Let
$\stab_{i_0},\stab_{i_1}=\iso_\stab(\stab_{i_0}),\dotsc,\stab_{i_{m-1}}=\iso_\stab^{m-1}(\stab_{i_0})$
be distinct components of $\stab$,
$\iso_\stab^m(\stab_{i_0})=\stab_{i_0}$ and assume that $m\geq 2$.
Most of the proof of Proposition~3.8.\ in~\cite{lu2007} applies to
the case of Prym curves and implies that the normalization
$\stab_{i_0}^\nu$ is rational and there are exactly three preimages
of nodes $p_1^+, p_2^+ , p_3^+\in C_{i_0}^\nu$  mapping to different
nodes of $\stab$. By \cite{lu2007} Proposition 3.7 at least one of
$p_1$, $p_2$, $p_3$ is fixed by $\iso_\stab$. If either one or all
three nodes are fixed, then $g(\stab)=2$, impossible. Thus two
nodes, say $p_1$ and $p_2$, are fixed by $\iso_\stab$ while $p_3$ is
interchanged with a fourth node $p_4$. Interchanging $p_3$ and $p_4$
gives a contribution of $\frac 12$ to $\mbox{age}(\sigma, \zeta_n)$.
Now consider the action of $\iso_\stab$ near $p_1$ and let $xy=0$ be
a local equation of $\stab$ at $p_1$. We have that $t_1=xy\mapsto
yx=t_1$ and $\tau_1\mapsto\pm\tau_1$, where the minus sign is only
possible if $p_1\in N$. Since $p_1$ is not an elliptic tail node and
$\bigl(\prym,\iso\bigr)$ is singularity reduced, we have
$\tau_1\mapsto-\tau_1$, which gives an additional contribution of
$\frac 12$ to the age, that is, $\mathrm{age}(\iso,
\zeta_n)\geq\frac{1}{2}+\frac{1}{2}=1$, contradicting $(\star)$.
\end{proof}
\begin{proposition}[{\cite{HM} p. 28, 36, \cite{lu2007} Proposition 3.9}]\label{prop:comps}
We assume that $(\star)$ holds and denote by
$\varphi_j=\iso^\nu_{|\compnu}$  the induced automorphism of the
normalization $\compnu$ of the irreducible component $\comp$ of
$\stab$. Then the pair $\bigl(\compnu,\varphi_j\bigr)$ is one of the
following types:
\begin{enumerate}
\item $\varphi_j=\mathrm{Id}_{\compnu}$ and $\compnu$ arbitrary.
\item $\compnu$ is rational and $\ord(\varphi_j)=2,4$.
\item $\compnu$ is elliptic and $\ord(\varphi_j)=2,4,3,6$.
\item $\compnu$ is hyperelliptic of genus $2$ and $\varphi_j$ is the hyperelliptic
involution.
\item $\compnu$ is hyperelliptic of genus $3$ and $\varphi_j$ is the hyperelliptic
involution.
\item $\compnu$ is bielliptic of genus $2$ and $\varphi_j$ is the associated
involution.
\end{enumerate}
\end{proposition}
\noindent The possibility of $\sigma_C$ interchanging two nodes does
not appear, cf. \cite{lu2007} Prop. 3.10:
\begin{proposition}
Under the assumption $(\star)$, the automorphism $\iso_\stab$ fixes
all the nodes of $\stab$.
\end{proposition}
\begin{proposition}\label{prop:compdet}
Assume $(\star)$ holds. Let $\comp$ be a component of $\stab$ with
normalization $\compnu$, $D_j$ the divisor of the marked points on
$\compnu$ and $\varphi_j=\iso^\nu_{|\compnu}$. Then $(\compnu, D_j,
\varphi_j)$ is of one of the following types and the contribution to
 $\mathrm{age}(\sigma, \zeta_n)$ coming from
$H^1\big(\compnu,T_{\compnu}(-D_j)\big)\subset\lud{\upsilon}$ is at
least the following quantity $w_j$:
\begin{enumerate}
\item Identity component:
$\varphi_j=\mathrm{Id}_{\compnu}$, arbitrary pair $(\compnu,D_j)$
and $w_j=0$
\item Elliptic tail:  $\compnu$ is elliptic, $D_j=p_1^+$ and $p_1^+$ is
 fixed by~$\varphi_j$.
\begin{description}
\item[\textnormal{\textit{order 2}}] $\ord(\varphi_j)=2$ and  $w_j=0$
\item[\textnormal{\textit{order 4}}] $\compnu$ has $j$-invariant $1728$, $\ord(\varphi_j)=4$ and $w_j=\frac 12$
\item[\textnormal{\textit{order 3, 6}}] $\compnu$ has $j$-invariant $0$, $\ord(\varphi_j)=3$ or $6$ and $w_j=\frac 13$
\end{description}
\item Elliptic ladder: $\compnu$ is elliptic, $D_j=p_1^++p_2^+$,
with $p_1^+$ and $p_2^+$ both fixed by $\varphi_j$.
\begin{description}
\item[\textnormal{\textit{order 2}}] $\ord(\varphi_j)=2$ and $w_j=\frac 12$
\item[\textnormal{\textit{order 4}}] $\compnu$ has $j$-invariant $1728$, $\ord(\varphi_j)=4$ and $w_j=\frac 34$
\item[\textnormal{\textit{order 3}}] $\compnu$ has $j$-invariant $0$, $\ord(\varphi_j)=3$ and $w_j=\frac 23$
\end{description}
\item Hyperelliptic tail: $\compnu$ has genus $2$, $\varphi_j$ is the hyperelliptic involution,
 $D_j$ is of the form $D_j=p_1^+$ with $p_1^+$ fixed by $\varphi_j$ and $w_j=\frac 12$.
\end{enumerate}
\end{proposition}
\begin{proof}
The proof is along the lines of the proof of Proposition~3.11\ in
\cite{lu2007}. The only difference occurs in the case of a singular
elliptic tail on which $\iso$ acts with order $2$. Assume that
$\compnu$  is rational, $D_j=p_1^++p_1^-+p_2$, with
$\mbox{ord}(\varphi_j)=2$ which fixes $p_2^+$ and interchanges
$p_1^+$ and $p_1^-$. If $xy=0$ is an equation for $\stab$ at $p_1$,
then $\iso_\stab$ acts via $t_1=xy\mapsto yx=t_1$. Since $p_1$ is
not an elliptic tail node and $\bigl(\prym,\iso\bigr)$ is
singularity reduced, the node $p_1$ must be exceptional and
$\sigma\cdot \tau_1=-\tau_1$.

A deformation of $\prym$ over the locus $(\tau_i=0)_{i\neq 1}\subset
\mathbb C_{\tau}^{3g-3}$ smooths $p_1$. Furthermore, $\iso$ deforms
to an automorphism $\iso'$ of a general Prym curve
$(\prymi',\prymii',\prymiii')$ over this locus, $\varphi_j$ deforms
to the  involution $\varphi_j'$ on the smooth elliptic tail $\comp'$
such that it fixes the line bundle $\prymii'_{\comp'}$, and the
restrictions of $\iso$ and $\iso'$ to the complement of $\comp$
resp.\ $\comp'$ coincide. Over the non-exceptional subcurve
$\nonex\subset X$ we have
$(\widetilde\iso')^*\widetilde{\prymii}'\cong\widetilde{\prymii}'$.
Thus $\sigma \cdot \tau_1=\tau_1$ which is a contradiction. The case
of a singular elliptic tail is thus excluded.
\end{proof}
\begin{proposition}
Under the hypothesis $(\star)$, the hyperelliptic tail case does not
occur.
\end{proposition}
\begin{proof}
Let $\comp$ be a genus $2$ tail of $\stab$ and $\compi$ the second
component through $p_1$. The action of $\iso$ on
$H^1(\compnu,T_{\compnu}(-D_j))$ contributes $\frac 12$ to the age
of $\sigma$ and $\compi$ has to be one of the cases of
Proposition~\ref{prop:compdet}. If $\compi$ is elliptic, then
$g(\stab)=3$. If $\compi$ is a hyperelliptic tail or an elliptic
ladder, the action on $H^1(\compinu,T_{\compinu}(-D_{j'}))$
contributes  at least $\frac 12$. Therefore $\compi$ is an identity
component. If $xy=0$ is an equation for $\stab$ at $p_1$, then
$\iso_\stab$ acts via $t_1=xy\mapsto-xy=-t_1$. The node $p_1$ is
disconnecting, hence non-exceptional, and it is not an elliptic tail
node. Therefore, $\upsilon_1=\tau_1=t_1$ and $\iso$ acts as
$\sigma\cdot \upsilon_1 =-\upsilon_1$. This gives an additional
contribution of $\frac 12$ to the age of $\sigma$ finishing the
proof.
\end{proof}
\begin{proposition}
In situation $(\star)$ the elliptic ladder cases do not occur.
\end{proposition}
\begin{proof}
Let $\comp$ be an elliptic ladder of $\stab$ of order
$n_j=\ord(\varphi_j)$ and denote by $\compi$ resp.\ $\compii$ the
second component through the node $p_1$ resp.\ $p_2$. Since every
elliptic ladder contributes at least $\frac 12$ to the age, $\compi$
and $\compii$ can only be elliptic tails or identity components. If
both are elliptic tails, then $g(\stab)=3$, hence we may assume that
$\compi$ is an identity component. If $xy=0$ is an equation for
$\stab$ at $p_1$, then $\iso_\stab$ acts as $x\mapsto x$,
$y\mapsto\alpha y$ and $t_1\mapsto\alpha t_1$, where $\alpha$ is a
primitive $n_j$-th root of $1$. If $p_1$ is non-exceptional then
$\upsilon_1=\tau_1=t_1$ and the space
$H^1(\compnu,T_{\compnu}(-D_j))\oplus \mathbb C\cdot \upsilon_1$
contributes to the age at least
\[
1=\begin{cases}
\frac 12+\frac 12 & \text{if }n_j=2\\
\frac 34+\frac 14 & \text{if }n_j=4\\
\frac 23+\frac 13 & \text{if }n_j=3
\end{cases}
\]
Therefore $p_1\in N$. Since $N\subset \Gamma(\stab)$ is an eulerian
subgraph, the node $p_2$ is also exceptional, both $p_1$ and $p_2$
are non-disconnecting and $\compii$ is an identity component as
well. Moreover $\sigma_C\cdot t_i =\alpha t_i$, $i=1,2$. Since
$\upsilon_i=\tau_i$ and $\tau_i^2=t_i$ for $i=1,2$, we find that
$\sigma\cdot \upsilon_i=\alpha_i\upsilon_i$, $i=1,2$, where
$\alpha_i$ is a square root of $\alpha$. Therefore, the contribution
to the age of $\sigma$ coming from
$H^1(\compnu,T_{\compnu}(-D_j))\oplus \mathbb C\cdot
\upsilon_1\oplus \mathbb C\cdot \upsilon_2$ is at least
\[
1=\begin{cases}
\frac 12+\frac 14+\frac 14 & \text{if }n_j=2\\
\frac 34+\frac 18+\frac 18 & \text{if }n_j=4\\
\frac 23+\frac 16+\frac 16 & \text{if }n_j=3
\end{cases}
\]
and the case of elliptic ladders is excluded.
\end{proof}
\begin{proposition}
Under hypothesis  $(\star)$,  the case of an elliptic tail of order
$4$ does not occur.
\end{proposition}
\begin{proof}
Let $\comp$ be an elliptic tail of order $4$ and $\compi$ another
component of $C$ through $p_1$. Then $\sigma_{C |
C_j'}=\mbox{Id}_{C_j'}$ and $\iso_\stab$ acts as
$t_1=xy\mapsto\zeta_4xy=\zeta_4t_1$ for a suitable fourth root
$\zeta_4$ of $1$. Since $p_1$ is an elliptic tail node, we have
 $\upsilon_1=t_1^2$ and $\iso\cdot v_1=-v_1$. The
action of $\iso$ on $H^1(\compnu,T_{\compnu}(-D_j))\oplus \mathbb
C\cdot \upsilon_1$ contributes  $\geq \frac 12+\frac 12=1$ to
$\mbox{age}(\iso, \zeta_4)$ excluding this case.
\end{proof}
\begin{proposition}
In situation $(\star)$ there has to be at least one elliptic tail of
order $3$ or $6$.
\end{proposition}
\begin{proof}
Assume to the contrary that every component of $\stab$ is either an
identity component or an elliptic tail of order $2$. The action of
$\iso$ on every space $H^1(\compnu,T_{\compnu}(-D_j))$ is trivial.
If $p_1$ is the node of an elliptic tail of order $2$, then
$\sigma_C\cdot t_1=-t_1$ and we have $\upsilon_1=\tau_1^2=t_1^2$ and
$\iso\cdot \upsilon_1=\upsilon_1$. In case $p_1$ is non-exceptional
but not an elliptic tail node, $\iso_\stab\cdot t_1=t_1$. Since
$\upsilon_1=\tau_1=t_1$, we find that $\iso$ fixes $\upsilon_1$. If
$p_1\in N$, then $\iso_\stab\cdot t_1=t_1$ and
$\upsilon_1^2=\tau_1^2=t_1$ and $\iso$ acts as
$\upsilon_1\mapsto\pm\upsilon_1$. Since $\mbox{age}(\sigma,
\zeta_n)<1$, there is exactly one node $p_1$ such that $\sigma\cdot
\upsilon_1=-\upsilon_1$, that is, $\iso$ acts as quasi-reflection on
$\lud{\upsilon}$,  a contradiction.
\end{proof}
\begin{proof}[Proof of the only-if-part of Theorem~\ref{thm:cansing}]
We proved, that if $\bigl(\prym,\iso\bigr)$
 is a singularity reduced pair and  $\mathrm{age}(\sigma,
\zeta_n)<1$, where $n=\mbox{ord}(\sigma)$, there exists an elliptic
tail $C_j\subset C$ with $\mbox{Aut}(C_j)=\mathbb Z_6$ such that
$\mbox{ord}(\sigma_{C_j})\in \{3, 6\}$. Since
$\iso_{\comp}^*(\prymii_{\comp})\cong\prymii_{\comp}$, we find that
$\prymii_{\comp}=\OO_{\comp}$. Let $\bigl(\prym,\iso\bigr)$ be a
pair consisting of a Prym curve and an automorphism such that the
$\mbox{age}(\sigma, \zeta_n)<1$. By Proposition~\ref{prop:pair} we
may deform $(\prym,\iso)$ to a singularity reduced pair
$((\prymi',\prymii',\prymiii'),\iso')$ such that the actions of
$\iso$ on $\lud{\upsilon}$ and $\iso'$ on $\lud{\upsilon'}$ have the
same ages. Therefore $\prymi'$ has an elliptic tail $\comp'$ with
$\mbox{Aut}(\comp')=\mathbb Z_6$ such that $\prymii'_{\comp'}$ is
trivial and $\iso'$ acts on $\comp'$ of order $3$ or $6$. In the
deformation of $\prym$ to $(\prymi',\prymii',\prymiii')$ elliptic
tails are preserved  hence $\bigl((X, \eta, \beta), \sigma\bigr)$
enjoys the same properties.
\end{proof}
\begin{remark}
If $\sigma\in \mbox{Aut}(X, \eta, \beta)$ satisfies the inequality
$\mathrm{age}(\sigma, \zeta_n)<1$ (with respect to the action on
$\lud{\upsilon}$), then $\iso$ is an elliptic tail automorphism and
$\mbox{ord}(\sigma)\in \{3, 6\}$. Indeed, we already know that
$\iso_\stab\in \mbox{Aut}(C)$ acts with order $3$ or $6$ on an
elliptic tail $\comp$. The action of $\iso$ on
$H^1(\compnu,T_{\compnu}(-D_j))$ and the $\upsilon$-coordinate
corresponding to the elliptic tail node on $\comp$ contributes at
least $\frac 23$ to $\mathrm{age}(\sigma, \zeta_n)$. Thus there is
exactly one elliptic tail of order $3$ or $6$ and $\iso_\stab$ is an
elliptic tail automorphism of the same order. If $\iso$ is not an
 elliptic tail automorphism of $\prymi$, then there
 exists an exceptional component $E_1\subset \prymi$ on which $\iso$ acts non-trivially.
 Since $E_1$ connects two non-exceptional components of $\prymi$ on which $\iso$ acts
  trivially, $\sigma \cdot \upsilon_1=-\upsilon_1$, giving a
   contribution of $\frac 12$ and an age $\geq \frac 23+\frac 12\geq 1$.
\end{remark}

\noindent \emph{Proof of Theorem \ref{thm:canforms}.} We start with
a pluricanonical form $\omega$ on $\rr_g^{\mathrm{reg}}$ and show
that $\omega$ lifts to a desingularization of a neighbourhood of
every point $[X, \eta, \beta]\in \rr_g$. We may assume that $[X,
\eta, \beta]$ is a general non-canonical singularity of $\rr_g$,
hence $X=C_1\cup_p C_2$, where $[C_1, p]\in \cM_{g-1, 1}$ is general
and $[C_2, p]\in \cM_{1, 1}$ has $j$-invariant $0$. Furthermore
$\eta_{C_2}=\OO_{C_2}$ and $\eta_1:=\eta_{C_1}\in
\mbox{Pic}^0(C_1)[2]$. We consider the pencil $\phi:\mm_{1,
1}\longrightarrow\rr_g$ given by $\phi[C', p]=[C'\cup_p C_1,
\eta_{C'}=\OO_{C'}, \eta_{C_1}=\eta_1]$.  Since $\phi(\mm_{1, 1})
\cap \Delta_0^{\mathrm{ram}}=\emptyset$, we imitate \cite{HM} pg.
41-44 and construct an \emph{explicit} open neighbourhood
$\rr_g\supset S\supset \phi(\mm_{1, 1})$ such that the restriction
to $S$ of $\pi:\rr_g\rightarrow\mm_g$ is an isomorphism and every
form $\omega\in H^0(\rr_g^{\mathrm{reg}},
K_{\rr_g^{\mathrm{reg}}}^{\otimes l})$ extends to a resolution
$\widehat{S}$ of $S$. For an arbitrary non-canonical singularity we
show that $\omega$ extends locally to a desingularizaton  along the
lines of \cite{lu2007} Theorem 4.1. \hfill $\Box$

\end{document}